\documentclass[reqno,11pt,a4paper,final]{amsart}
\usepackage[a4paper,left=30mm,right=30mm,top=30mm,bottom=30mm,marginpar=20mm]{geometry} 
\usepackage{amsmath}
\usepackage{amssymb}
\usepackage{amsthm}
\usepackage{amscd}
\usepackage{stmaryrd}
\usepackage{esint}
\usepackage[ansinew]{inputenc}
\usepackage{cite}
\usepackage{bbm}
\usepackage{xcolor}
\usepackage[english=american]{csquotes}
\usepackage[final]{graphicx}
\usepackage{hyperref}
\usepackage{calc}
\usepackage{mathptmx}
\usepackage{bm}
\usepackage{enumerate}
\usepackage[shortlabels]{enumitem}
\usepackage{transparent}
\usepackage{xr}

\numberwithin{equation}{section}

\newtheoremstyle{thmlemcorr}{10pt}{10pt}{\itshape}{}{\bfseries}{.}{10pt}{{\thmname{#1}\thmnumber{ #2}\thmnote{ (#3)}}}
\newtheoremstyle{thmlemcorr*}{10pt}{10pt}{\itshape}{}{\bfseries}{.}\newline{{\thmname{#1}\thmnumber{ #2}\thmnote{ (#3)}}}
\newtheoremstyle{remexample}{10pt}{10pt}{}{}{\bfseries}{.}{10pt}{{\thmname{#1}\thmnumber{ #2}\thmnote{ (#3)}}}

\theoremstyle{thmlemcorr}
\newtheorem{theorem}{Theorem}
\numberwithin{theorem}{section}
\newtheorem{lemma}[theorem]{Lemma}
\newtheorem{corollary}[theorem]{Corollary}
\newtheorem{proposition}[theorem]{Proposition}

\theoremstyle{thmlemcorr*}
\newtheorem{theorem*}{Theorem}
\newtheorem{lemma*}[theorem]{Lemma}
\newtheorem{corollary*}[theorem]{Corollary}
\newtheorem{proposition*}[theorem]{Proposition}
\newtheorem{problem*}[theorem]{Problem}
\newtheorem{conjecture*}[theorem]{Conjecture}
\newtheorem{definition*}[theorem]{Definition}

\theoremstyle{remexample}
\newtheorem{remark}[theorem]{Remark}
\newtheorem{example}[theorem]{Example}

\newcommand{\Crm}{\mathrm{C}}

\newcommand{\Irm}{\mathrm{I}}

\newcommand{\Lrm}{\mathrm{L}}

\newcommand{\Trm}{\mathrm{T}}

\newcommand{\Bcal}{\mathcal{B}}

\newcommand{\Dcal}{\mathcal{D}}

\newcommand{\Fcal}{\mathcal{F}}

\newcommand{\Hcal}{\mathcal{H}}

\newcommand{\Lcal}{\mathcal{L}}
\newcommand{\Mcal}{\mathcal{M}}

\newcommand{\Fbf}{\mathbf{F}}

\newcommand{\Jbf}{\mathbf{J}}

\newcommand{\Mbf}{\mathbf{M}}

\newcommand{\Pbf}{\mathbf{P}}

\newcommand{\Fbb}{\mathbb{F}}

\DeclareMathOperator*{\esssup}{ess\,sup}

\DeclareMathOperator{\Id}{Id}

\DeclareMathOperator{\graph}{graph}

\DeclareMathOperator*{\wslim}{w*-lim}

\DeclareMathOperator{\dist}{dist}

\DeclareMathOperator{\tr}{tr}
\DeclareMathOperator{\spn}{span}

\DeclareMathOperator{\supp}{supp}

\DeclareMathOperator{\Wedge}{{\textstyle\bigwedge}}

\newcommand{\ee}{\mathrm{e}}

\newcommand{\setb}[2]{\bigl\{\, #1 \ \ \textup{\textbf{:}}\ \ #2 \,\bigr\}}
\newcommand{\setB}[2]{\Bigl\{\, #1 \ \ \textup{\textbf{:}}\ \ #2 \,\Bigr\}}
\newcommand{\setBB}[2]{\biggl\{\, #1 \ \ \textup{\textbf{:}}\ \ #2 \,\biggr\}}

\newcommand{\setsub}[2]{\{ #1 \ \textup{:}\ #2 \}}
\newcommand{\norm}[1]{\|#1\|}

\newcommand{\abs}[1]{|#1|}

\newcommand{\absb}[1]{\bigl|#1\bigr|}

\newcommand{\absBB}[1]{\biggl|#1\biggr|}

\newcommand{\altnorm}[1]{{\left\vert\kern-0.25ex\left\vert\kern-0.25ex\left\vert #1 \right\vert\kern-0.25ex\right\vert\kern-0.25ex\right\vert}}

\newcommand{\dpr}[1]{\langle #1 \rangle}

\newcommand{\dprb}[1]{\bigl\langle #1 \bigr\rangle}

\newcommand{\dprBB}[1]{\biggl\langle #1 \biggr\rangle}

\newcommand{\dbr}[1]{\llbracket #1 \rrbracket}

\newcommand{\cl}[1]{\overline{#1}}

\newcommand{\dd}{\;\mathrm{d}}

\newcommand{\N}{\mathbb{N}}
\newcommand{\R}{\mathbb{R}}

\newcommand{\Z}{\mathbb{Z}}
\newcommand{\loc}{\mathrm{loc}}

\newcommand{\ONE}{\mathbbm{1}}

\newcommand{\toweakstar}{\overset{*}\rightharpoonup}

\newcommand{\todown}{\downarrow}

\newcommand{\BigO}{\mathrm{\textup{O}}}

\newcommand{\sbullet}{\begin{picture}(1,1)(-0.5,-2.5)\circle*{2}\end{picture}}
\newcommand{\frarg}{\,\sbullet\,}
\newcommand{\BV}{\mathrm{BV}}

\newcommand{\eps}{\epsilon}

\newcommand{\tv}[1]{\norm{#1}}

\newcommand{\term}[1]{\textbf{#1}}

\newcommand{\Diss}{\mathrm{Diss}}

\newcommand{\tbf}{\mathbf{t}}
\newcommand{\pbf}{\mathbf{p}}

\newcommand{\Lip}{\mathrm{Lip}}
\newcommand{\FpV}{\mathrm{pV}_\Fbf}

\DeclareMathOperator{\Tan}{T}

\DeclareMathOperator{\Var}{Var}

\DeclareMathOperator{\cone}{\lhd}
\DeclareMathOperator{\revcone}{\rhd}

\newcounter{assumption}
\makeatletter
\newcommand{\nextas}[1]{%
   \refstepcounter{assumption}%
   \protected@write \@auxout{}{\string\newlabel{#1}{{(A\theassumption)}{\thepage}{(A\theassumption)}{#1}{}}}%
   \hypertarget{#1}{(A\theassumption)}%
}
\newcommand{\nextasnamed}[2]{%
   \refstepcounter{assumption}%
   \protected@write \@auxout{}{\string\newlabel{#1}{{(#2)}{\thepage}{(#2)}{#1}{}}}%
   \hypertarget{#1}{(#2)}%
}
\makeatother

\def\XXint#1#2#3{{\setbox0=\hbox{$#1{#2#3}{\int}$} 
\vcenter{\hbox{$#2#3$}}\kern-.5\wd0}}

\usepackage{pict2e}
\makeatletter
\DeclareRobustCommand{\intprod}{%
  \mathbin{\mathpalette\int@prod{(0.1,0)(0.9,0)(0.9,0.8)}}}
\DeclareRobustCommand{\restrict}{%
  \mathbin{\mathpalette\int@prod{(0.1,0.8)(0.1,0)(0.9,0)}}}	
\newcommand{\int@prod}[2]{%
  \begingroup
  \sbox\z@{$\m@th#1+$}%
  \setlength\unitlength{\wd\z@}%
  \begin{picture}(1,1)
  \roundcap
  \polyline#2
  \end{picture}%
  \endgroup
}
\makeatother

\renewcommand{\eps}{\varepsilon}
\renewcommand{\epsilon}{\varepsilon}
\renewcommand{\phi}{\varphi}
\renewcommand{\tilde}{\widetilde}
\renewcommand{\hat}{\widehat}
\renewcommand{\bar}{\overline}

\begin{document}


\title[BV integral currents]{Space-time integral currents of bounded variation}


\author{Filip Rindler}
\address{Mathematics Institute, University of Warwick, Coventry CV4 7AL, United Kingdom.}
\email{F.Rindler@warwick.ac.uk}


\hypersetup{
  pdfauthor = {Filip Rindler},
  pdftitle = {Space-time integral currents of bounded variation}
}


\maketitle


\begin{abstract}
Motivated by a recent model for elasto-plastic evolutions that are driven by the flow of dislocations, this work develops a theory of space-time integral currents with bounded variation in time, which enables a natural variational approach to the analysis of rate-independent geometric evolutions. Based on this, we further introduce the notion of Lipschitz deformation distance between integral currents, which arises physically as a (simplified) dissipation distance. Several results are obtained: A Helly-type compactness theorem, a deformation theorem, an isoperimetric inequality, and the equivalence of the convergence in deformation distance with the classical notion of weak* (or flat) convergence. Finally, we prove that the Lipschitz deformation distance agrees with the (integral) homogeneous Whitney flat metric for boundaryless currents. Physically, this means that two seemingly different ways to measure the dissipation actually coincide.
 
\vspace{4pt}




\noindent\textsc{Date:} \today{}.
\end{abstract}



\section{Introduction}

The recent work~\cite{HudsonRindler21?} introduced a model for the evolution of macroscopic plastic deformations of single crystals based on the movement of dislocations, that is, $1$-dimensional topological defects in the crystal lattice~\cite{AbbaschianReedHill09,HullBacon11book,AndersonHirthLothe17book}. Like in a number of previous works, see, e.g.,~\cite{ContiGarroniMassaccesi15,ContiGarroniOrtiz15,ScalaVanGoethem19}, these dislocations are modelled mathematically as $1$-dimensional boundaryless integral currents~\cite{KrantzParks08book,Federer69book}. An existence result for energetic solutions to the model in~\cite{HudsonRindler21?} is established in~\cite{Rindler21b?}, for which the theory developed in the present work is an essential ingredient.

A pivotal feature of the approach in~\cite{HudsonRindler21?} is that the evolution of all dislocations with a Burgers vector $b$ (from a finite set $\Bcal$ of all possible Burgers vectors) is represented as a \emph{slip trajectory}, namely a $2$-dimensional integral current $S^b$ in the space-time cylinder $[0,T] \times \R^3$ with the property that
\begin{equation} \label{eq:partialSb}
  \partial S^b \restrict ((0,T) \times \R^3) = 0.
\end{equation}
Then, the dislocations with Burgers vector $b$ at time $t \in (0,T)$ are given by
\[
  T^b(t) := \pbf_*(S^b|_t),
\]
that is, the pushforward under the spatial projection $\pbf(t,x) := x$ of the slice $S^b|_t$ of $S^b$ at time $t$ (more precisely, the slice with respect to the temporal projection $\tbf(t,x) := t$). The theory of integral currents in conjunction with~\eqref{eq:partialSb} entails that $T^b(t)$ is a $1$-dimensional integral current and $\partial T^b(t) = 0$ for almost every $t \in (0,T)$.

The $2$-current in space given as
\[
  S^b|_s^t := \pbf_* \bigl[ S^b \restrict ([s,t] \times \R^3) \bigr]
\]
is the traversed surface from $T^b(s)$ to $T^b(t)$. Customarily, $S^b|_s^t$ is called the \emph{slip surface} from $s$ to $t$. The principal reason for our space-time approach, which employs the full slip trajectory and not just the slip surface, is that the evolution of the plastic distortion in the full model requires a \enquote{time index} (i.e., the time coordinate) along the dislocation movement. Also, $S^b$ may traverse an area multiple times with different orientations, which leads to cancellation in $S^b|_s^t$, but of course these multiply-traversed areas all have to be counted in the dissipation (that is, the energetic cost of dislocation movement) with their correct \enquote{absolute} multiplicity.  We remark that in the context of gradient flows the benefit of space-time currents was also noticed independently in~\cite{Kampschulte17PhD}. 

For the rate-independent case of the model in~\cite{HudsonRindler21?}, the dissipational cost as the dislocations (with Burgers vector $b$) move from time $s$ to time $t$, is derived in~\cite{HudsonRindler21?,Rindler21b?} (using the $2$-vector version of the geometric slip rate) to be of the form 
\begin{equation} \label{eq:Diss_intro}
  \Diss(S^b;[s,t]) := \int_{[s,t] \times \R^3} R^b\bigl( P(\tau,x) \pbf(\vec{S}^b(\tau,x)) \bigr) \dd \tv{S^b}(\tau,x).
\end{equation}
Here, the function $R^b \colon \Wedge_2\R^3  \to [0,\infty)$ is the (possibly anisotropic and $b$-dependent) convex and $1$-homogeneous \emph{dissipation potential}, which expresses the dissipational cost of a unit slip trajectory, and $P(\tau,x)$ denotes the \emph{plastic distortion} at time $\tau$ and point $x$, which keeps track of how much the specimen has deformed plastically at $x$ up to the time $\tau$. Moreover, we have denoted by $S^b = \vec{S}^b \, \tv{S^b}$ the Radon--Nikod\'{y}m decomposition of the integral current $S^b$ into its \emph{orienting $2$-vector} $\vec{S}^b \in \Lrm^\infty(\tv{S^b};\Wedge_2 \R^{1+3})$ (which is simple and has unit length) and the \emph{total variation measure} $\tv{S^b} \in \Mcal^+([0,T] \times \R^3)$. Since we assumed $S^b$ to be an \emph{integral} $2$-current, $\tv{S^b} = m \, \Hcal^2 \restrict R$ with the \emph{multiplicity} $m \in \Lrm^1(\Hcal^2 \restrict R;\N)$ and $R$ a (countably) $\Hcal^2$-rectifiable set. The applications of $\pbf$ and $P$ in~\eqref{eq:Diss_intro} are understood as the pushforwards of a $2$-vector under $\pbf$ and $P$, respectively, namely $P(v \wedge w) = (Pv) \wedge (Pw)$ and for non-simple $2$-vectors extended by linearity; likewise for $\pbf$. We refer to Section~\ref{sc:notation} below for details on these notions. 


It turns out that the dissipation given in~\eqref{eq:Diss_intro} controls a type of \emph{variation} of $S^b$ in the interval $[s,t]$, namely
\[
  \Var(S^b;[s,t]) := \int_{[s,t] \times \R^3} \abs{\pbf(\vec{S}^b)} \dd \tv{S^b}.
\]
Indeed, if we make the natural assumption $C^{-1}\abs{\xi} \leq R^b(\xi) \leq C\abs{\xi}$ ($C > 0$ independent of $b$), then, for $P$ close to the identity,
\[
  \Diss(S^b;[s,t]) \geq C^{-1} \cdot \Var(S^b;[s,t]),
\]
and the claimed coercivity holds. In the case where $P$ is not near the identity, we need to modify the dissipation from the form given in~\eqref{eq:Diss_intro} by multiplying it with a hardening factor, which depends on (the cofactors of) $P$. We omit the details of this modification here and refer to the forthcoming work~\cite{Rindler21b?} for precise assumptions and the resulting coercivity estimate.

Generalizing to higher-dimensional integral currents, we define the variation of a $(1+k)$-dimensional integral current $S$ in the space-time cylinder $[0,T] \times \R^d$ as
\[
  \Var(S;I) := \int_{I \times \R^d} \abs{\pbf(\vec{S})} \dd \tv{S}
\]
for any interval $I \subset [0,T]$. We take this definition as the starting point of a \emph{BV-theory of space-time currents}, which turns out to be more natural than the classical approach via BV-functions on a time interval with values in the space of integral $k$-current (metrized with the flat norm). It can also be seen easily, and we will do so in Example~\ref{ex:BV}, that classical functions of bounded variation~\cite{AmbrosioFuscoPallara00book} constitute the case $k = 0$. However, even in this case the present theory is stronger, in particular allowing one to express the path connecting jump endpoints (more akin to BV-liftings~\cite{JungJerrard04,RindlerShaw19} or cartesian currents~\cite{GiaquintaModicaSoucek98book1,GiaquintaModicaSoucek98book2}). While our main motivation is to lay the groundwork for the rigorous analysis in~\cite{Rindler21b?} of the model from~\cite{HudsonRindler21?}, this theory seems interesting in its own right and may be useful for other applications as well.

The first aim of the present work is thus to develop aspects of the general theory of BV space-time currents. Most notably, we will obtain a Helly-type compactness theorem (Theorem~\ref{thm:current_Helly}), a suitable deformation theorem (Theorem~\ref{thm:defo}), an isoperimetric inequality (Theorem~\ref{thm:isop}), and the equivalence of our notion of convergence with respect to a \enquote{deformation distance} (see below) with the classical weak* convergence (or Whitney flat norm convergence) of currents (Theorem~\ref{thm:equiv}).

The second aim of the present work is to answer the following question, which is compelling from both mathematical and physical perspectives: To measure the \enquote{distance} between two boundaryless integral $k$-currents $T_0,T_1$ (e.g., representing dislocation systems if $k=1$), we now have two options: Classically, one might measure this distance via the (integral) homogeneous Whitney flat norm (see, e.g.,~\cite{ScalaVanGoethem19} for such an approach in the theory of dislocations), i.e.,
\[
  \Fbb(T_1-T_0) := \inf \, \setB{ \Mbf(Q) }{ \text{$Q$ integral $(1+k)$-current with $\partial Q = T_1-T_0$} }.
\]
We remark that in this work we use exclusively the \emph{integral} versions of the flat norm, where all currents are assumed to be integral. In general dimensions it seems to be unknown if they are equal to their non-integral counterparts, see~\cite{HardtPitts86} or~\cite[Remark~5]{BrezisMironescu19}.
Alternatively, in the spirit of the theory developed in the present work, we could employ the \emph{deformation distance}
\begin{align*}
  \dist_{\Lip}(T_0,T_1) := \inf \setB{ \Var(S;[0,1]) }{ &\text{$S$ integral $(1+k)$-current in $[0,1] \times \R^d$ from $T_0$ to $T_1$,} \\
  &\text{that is, $\partial S = \delta_1 \times T_1 - \delta_0 \times T_0$, and $S$ Lipschitz in time}  }.
\end{align*}
We postpone the precise definition of the notion of Lipschitz regularity in time until Section~\ref{sc:Var}, but its main condition (in this case) is that $t \mapsto \Var(S;[0,t])$ is a scalar Lipschitz function.

The first option corresponds to measuring the dissipation as the area of the slip surface and the second option corresponds to measuring the dissipation as the variation of the slip trajectory, i.e.,~\eqref{eq:Diss_intro} for $R^b(\xi) = \abs{\xi}$ and $P = \Id$ (i.e., isotropically and \enquote{near the identity plastic deformation}). As mentioned before, only the space-time formulation provides us with the \enquote{time index} needed to define the evolution for the plastic distortion $P$. However, for considerations only pertaining to the dissipational cost of dislocation movement near the identity plastic deformation $P = \Id$, the existence of said time index should be irrelevant. Hence, neglecting issues of domains, one may conjecture that
\[
  \Fbb(T_1-T_0) = \dist_{\Lip}(T_0,T_1).
\]
The inequality \enquote{$\leq$} is obviously true via a pushforward under the spatial projection since the resulting slip surface is admissible in $\Fbb$. In fact, even the conjectured equality turns out to be true, see Theorem~\ref{thm:equal}, but the proof of the  inequality \enquote{$\geq$} is much more involved. As a consequence of this theorem, one can always find a \emph{time-indexed} minimizer for $\Fbb(T_1-T_0)$, which deforms $T_0$ to $T_1$ progressively (even with Lipschitz regularity in time). This fact is perhaps somewhat surprising and seems interesting beyond the motivation of the conjecture.

The outline of this paper is as follows: After recalling notation and basic facts in Section~\ref{sc:notation}, we introduce the theory of space-time integral currents of bounded variation (in time) in Section~\ref{sc:BVcurr}. We build on this is Section~\ref{sc:homdef} to define a deformation theory of currents. Finally, Section~\ref{sc:defdist} investigates the deformation distance and proves the above conjecture.

\subsection*{Acknowledgements}

This project has received funding from the European Research Council (ERC) under the European Union's Horizon 2020 research and innovation programme, grant agreement No 757254 (SINGULARITY). The author would like to thank Giovanni Alberti, Paolo Bonicatto, Antonio De~Rosa, Giacomo Del~Nin, Thomas Hudson, Fang-Hua Lin, Ulrich Menne, and Felix Schulze for discussions related to this work.

\section{Notation and preliminaries} \label{sc:notation}

In this section we fix our notation, collect some known results, and recall tools that will be needed later on.

\subsection{Linear and multilinear algebra}

If not stated otherwise, on the space of matrices $\R^{m \times n}$ we use the Frobenius inner product $A : B := \sum_{ij} A^i_j B^i_j = \tr(A^TB) = \tr(B^TA)$, where upper indices indicate rows and lower indices indicate columns. As matrix norm we use the induced Frobenius norm, i.e., $\abs{A} := (A : A)^{1/2} = (\tr(A^T A))^{1/2}$.

In all of the following, let $k = 0,1,2,\ldots,n$. We denote the set of $k$-vectors in an $n$-dimensional real Hilbert space $V \cong \R^n$ by $\Wedge_k V$ and the set of $k$-covectors in $V$ by $\Wedge^k V$ (in particular, $\Wedge_0 V \cong \Wedge^0 V \cong \R$ are the real scalars). Recall that a $k$-vector $\eta \in \Wedge_k V$ is called simple if $\eta = v_1 \wedge \cdots \wedge v_k$ for $v_\ell \in V$ ($\ell = 1,\ldots,k$), where \enquote{$\wedge$} denotes the exterior (wedge) product; likewise for $k$-covectors. The duality pairing between a simple $k$-vector $\xi = v_1 \wedge \cdots \wedge v_k$ and a simple $k$-covector $\alpha = w^1 \wedge \cdots \wedge w^k$ is given as $\dpr{\xi,\alpha} = \det \, (v_i \cdot w^j)^i_j$ and the duality product is then extended to non-simple $k$-vectors and $k$-covectors by linearity. For $\eta \in \Wedge_k V$ and $\alpha \in \Wedge^l V$ we define $\eta \intprod \alpha \in \Wedge^{l-k} V$ and $\eta \restrict \alpha \in \Wedge_{k-l} V$ via
\begin{align*}
  \dprb{\xi, \eta \intprod \alpha} := \dprb{\xi \wedge \eta, \alpha},  \qquad \xi \in \Wedge_{l-k} V,\\
  \dprb{\eta \restrict \alpha, \beta} := \dprb{\eta, \alpha \wedge \beta},  \qquad \beta \in \Wedge^{k-l} V.
\end{align*}
The mass and comass norms of $\eta \in \Wedge_k V$ and $\alpha \in \Wedge^k V$ are denoted by
\begin{align*}
  \abs{\eta} &:= \sup \setb{ \absb{\dprb{\eta,\alpha}} }{ \alpha \in \Wedge^k V,\;\abs{\alpha} = 1 },  \\
  \abs{\alpha} &:= \sup \setb{\absb{\dprb{\eta,\alpha}} } { \eta \in \Wedge_k V \text{ simple, unit} },
\end{align*}
respectively. Here, a simple $k$-vector $\eta$ is called a unit if it can be expressed as $\eta = v_1 \wedge \cdots v_k$ with the $v_i$ forming an orthonormal basis of $\spn \eta := \spn \{ v_1, \ldots, v_k \}$.

If $S \colon V \to W$ is linear, where $V,W$ are real finite-dimensional Hilbert spaces, we define the linear map $\Wedge^k S \colon \Wedge^k V \to \Wedge^k W$ by setting, for $v_1, \ldots, v_k \in V$,
\[
  S(v_1 \wedge \cdots \wedge v_k) = (\Wedge^k S)(v_1 \wedge \cdots \wedge v_k) := (Sv_1) \wedge \cdots \wedge (Sv_k)
\]
and extending by (multi-)linearity to $\Wedge^k V$. We will still usually write simply $S$ for $\Wedge^k S$.

\subsection{Area and coarea formula} \label{sc:area_coarea}

For the convenience of the reader (and for easy reference later), we recall the area and coarea formulas and refer to~\cite[3.2.22]{Federer69book},~\cite[Sections~2.10,~2.12]{AmbrosioFuscoPallara00book}, \cite[Chapter~5]{KrantzParks08book} for proofs. As usual, we denote by $\Hcal^k \restrict R$ the $k$-dimensional Hausdorff measure restricted to a (countably) $k$-rectifiable set; $\Lcal^d$ is the $d$-dimensional Lebesgue measure.

Let $R \subset \R^d$ be a countably $k$-rectifiable set with $\Hcal^k(R \cap K) < \infty$ for every compact set $K \subset \R^d$, and let $f \colon \R^d \to \R^m$ be Lipschitz continuous. For $\Hcal^k$-almost every $x \in R$ the approximate tangent space $\Trm_x R = \spn\{v_1, \ldots, v_k\}$ (with $\{v_i\}_i$ an orthonormal basis) and the restriction $D^R f(x)$ of the differential $Df(x)$ to $\Trm_x R$ exist. Moreover, we may identify $D^R f(x)$ with an $(m \times k$)-matrix (with respect to $\{v_i\}$). Then, define the \term{$k$-dimensional Jacobian $\Jbf^R_k f$ of $f$ relative to $R$} for $k \leq m$ via
\[
  \Jbf^R_k f(x) := \sqrt{\det(D^R f(x)^T D^R f(x))} = \absb{D^R f(x)[v_1 \wedge \cdots \wedge v_k]} = \absb{Df(x)[v_1] \wedge \cdots \wedge Df(x)[v_k]} \; ,
\]
and for $k \geq m$ define the \term{$m$-dimensional Jacobian $\Jbf^R_m f$ of $f$ relative to $R$} via
\[
  \Jbf^R_m f(x) := \sqrt{\det(D^R f(x)D^R f(x)^T)} \; .
\]
It is easy to see that the above formulas do not depend on the choice of the orthonormal basis $\{v_i\}$ (see, e.g.,~\cite[Lemma~5.3.5]{KrantzParks08book}).

\begin{proposition}[Area formula] \label{prop:area}
Let $R \subset \R^d$ be a countably $k$-rectifiable set with $\Hcal^k(R \cap K) < \infty$ for every compact set $K \subset \R^d$, and let $f \colon \R^d \to \R^m$ be Lipschitz continuous with $k \leq m$. Then, for every $\Hcal^k$-measurable map $g \colon R \to \R^N$, it holds that
\[
  \int_R g(x) \; \Jbf^R_k f(x) \dd \Hcal^k(x) = \int_{\R^m} \sum_{x \in R \cap f^{-1}(y)} g(x) \dd \Hcal^k(y).
\]
In particular, 
\[
  \int_R \Jbf^R_k f(x) \dd \Hcal^k(x) = \int_{\R^m} \Hcal^0(R \cap f^{-1}(y)) \dd \Hcal^k(y)
\]
and, if $f$ is injective,
\[
  \int_R g(x) \; \Jbf^R_k f(x) \dd \Hcal^k(x) = \int_{f(R)} g(f^{-1}(y)) \dd \Hcal^k(y), \qquad
  \int_R \Jbf^R_k f(x) \dd \Hcal^k(x) = \Hcal^k(f(R)).
\]
\end{proposition}

\begin{proposition}[Coarea formula] \label{prop:coarea}
Let $R \subset \R^d$ be a countably $k$-rectifiable set with $\Hcal^k(R \cap K) < \infty$ for every compact set $K \subset \R^d$, and let $f \colon \R^d \to \R^m$ be Lipschitz continuous with $k \geq m$. Then, for every $\Hcal^k$-measurable map $g \colon R \to \R^N$, it holds that
\[
  \int_R g(x) \; \Jbf^R_m f(x) \dd \Hcal^k(x) = \int_{\R^m} \int_{R \cap f^{-1}(y)} g(x) \dd \Hcal^{k-m}(x) \dd \Hcal^m(y).
\]
In particular,
\[
  \int_R \Jbf^R_m f(x) \dd \Hcal^k(x) = \int_{\R^m} \Hcal^{k-m}(R \cap f^{-1}(y)) \dd \Hcal^m(y).
\]
\end{proposition}

\subsection{Integral currents} \label{sc:curr}

Let us now recall some notions from the theory of currents, see~\cite{Federer69book,KrantzParks08book} for details and proofs. Denote by $\Dcal^k(U)$ ($k \in \N \cup \{0\}$) the space of \term{(smooth) differential $k$-forms} with compact support in an open set $U \subset \R^d$ (the ambient dimension $d$ being fixed), that is, $\Dcal^k(U) := \Crm^\infty_c(U;\Wedge^k \R^d)$, where $\Crm^\infty_c(U;W)$ contains all smooth maps that take values in the finite-dimensional normed vector space $W$ and that are compactly supported in $U$. The exterior differential of $\omega \in \Dcal^k(U)$ is denoted by $d\omega \in \Dcal^{k+1}(U)$.

The elements of the dual space $\Dcal_k(U) := \Dcal^k(U)^*$ are called \term{$k$-currents}. We define the \term{boundary} of a $k$-current $T \in \Dcal_k(\R^d)$, where now $k \geq 1$, as the $(k-1)$-current $\partial T \in \Dcal_{k-1}(\R^d)$ determined via
\[
  \dprb{\partial T, \omega} := \dprb{T, d\omega}, \qquad \omega \in \Dcal^{k-1}(\R^d).
\]
For a $0$-current $T$, we formally set $\partial T := 0$. 

In this work we will only deal with restricted subclasses of currents, namely the following: A (local) Borel measure $T \in \Mcal_\loc(\R^d;\Wedge_k \R^d)$ is called an \term{integer-multiplicity rectifiable $k$-current} if it is of the form
\[
  T = m \, \vec{T} \, \Hcal^k \restrict R,
\]
that is,
\[
  \dprb{T,\omega} = \int_R \dprb{\vec{T}(x), \omega(x)} \, m(x) \dd \Hcal^k(x),  \qquad \omega \in \Dcal^k(\R^d),
\]
where
\begin{enumerate}[(i)]
	\item $R \subset \R^d$ is $\Hcal^k$-rectifiable with $\Hcal^k(R \cap K) < \infty$ for all compact sets $K \subset \R^d$;
	\item $\vec{T} \colon R \to \Wedge_k \R^d$ is $\Hcal^k$-measurable and for $\Hcal^k$-a.e.\ $x \in R$ the $k$-vector $\vec{T}(x)$ is simple, has unit length ($\abs{\vec{T}(x)} = 1$), and spans the approximate tangent space $\Tan_x R$ to $R$ at $x$;
	\item $m \in \Lrm^1_\loc(\Hcal^k \restrict R;\N)$;
\end{enumerate}
One calls $\vec{T}$ the \term{orientation map} of $T$ and $m$ the \term{multiplicity}.

We denote by $\tv{T} := m \, \Hcal^k \restrict R \in \Mcal^+_\loc(\R^d)$ the \term{total variation measure} of $T$, so that $T = \vec{T} \tv{T}$ is the Radon--Nikod\'{y}m decomposition of $T$ (considered as a measure). The \term{(global) mass} of $T$ is
\[
  \Mbf(T) := \tv{T}(\R^d)
  = \sup_{\substack{\omega \in \Dcal^k(\R^d)\\\abs{\omega} \leq 1}} \int_{\R^d} \dprb{\vec{T}, \omega} \dd \tv{T}
  = \int_R m(x) \dd \Hcal^k(x).
\]
The \term{support} $\supp T$ of $T$ is the support in the sense of measures.

Here and in all of the following, let $\Omega \subset \R^d$ be a bounded Lipschitz domain, i.e., open, connected and with a (strong) Lipschitz boundary. The members of the following sets are called \term{integral $k$-currents} ($k \in \N \cup \{0\}$):
\begin{align*}
  \Irm_k(\R^d) &:= \setb{ \text{$T$ integer-multiplicity rectifiable $k$-current} }{ \Mbf(T) + \Mbf(\partial T) < \infty }, \\
  \Irm_k(\cl{\Omega}) &:= \setb{ T \in \Irm_k(\R^d) }{ \supp T \subset \cl{\Omega} }.
\end{align*}
By the boundary rectifiability theorem, see~\cite[4.2.16]{Federer69book} or~\cite[Theorem~7.9.3]{KrantzParks08book}, for $T \in \Irm_k(\R^d)$ it holds that $\partial T \in \Irm_{k-1}(\R^d)$.  

\begin{remark} \label{rem:retract}
In Federer's language~\cite[4.1.29]{Federer69book}, $\cl{\Omega}$ is a compact \emph{Lipschitz neighborhood retract}, i.e., there exists a Lipschitz map that retracts some neighborhood of $\cl{\Omega}$ onto $\cl{\Omega}$. In fact, since $\Omega$ was assumed to be bounded and to have a (strong) Lipschitz boundary, one may proceed by observing that $\Omega$ is a Lipschitz manifold and thus a Lipschitz neighborhood retract, see, e.g.,~\cite[Theorem~5.13 and Remark~3.2~(3)]{LuukkainenVaisala77}).
\end{remark}

Let $T_1 = m_1 \, \vec{T}_1 \, \Hcal^{k_1} \restrict R_1 \in \Irm_{k_1}(\R^{d_1})$ and $T_2 = m_2 \, \vec{T}_2 \, \Hcal^{k_2} \restrict R_2 \in \Irm_{k_2}(\R^{d_2})$ with $R_1$ $k_1$-rectifiable (not just $\Hcal^{k_1}$-rectifiable) or $R_2$ $k_2$-rectifiable, so that the product set $R_1 \times R_2$ is $\Hcal^{k_1+k_2}$-rectifiable. Then, the \term{product current} of $T_1,T_2$ is
\[
  T_1 \times T_2 := m_1 m_2 \, (\vec{T}_1 \wedge \vec{T}_2) \, \Hcal^{k_1+k_2} \restrict (R_1 \times R_2) \in \Irm_{k_1+k_2}(\R^{d_1+d_2}).
\]
For its boundary we have the formula
\begin{equation} \label{eq:product_bdry}
  \partial(T_1 \times T_2) = \partial T_1 \times T_2 + (-1)^{k_1} T_1 \times \partial T_2.
\end{equation}
In particular, for $T_1 = \dbr{(0,1)}$, i.e., the canonical current associated to the interval $(0,1)$ (with orientation $+1$ and multiplicity $1$), and $T_2 = T \in \Irm_k(\R^d)$,
\[
  \partial(\dbr{(0,1)} \times T) = \delta_1 \times T - \delta_0 \times T - \dbr{(0,1)} \times \partial T,
\]
where $\delta_z$ as usual denotes the Dirac point mass at $z$, here understood as a $0$-dimensional integral current.

We also recall briefly the notion of pushforwards. Let $\theta \colon \cl{\Omega} \to \R^m$ be smooth and let $\theta|_{\supp T}$ be proper, i.e., $\theta^{-1}(K) \cap \supp T$ is compact for every compact $K \subset \R^m$. Further, let $T = m \, \vec{T} \, \Hcal^k \restrict R \in \Irm_k(\cl{\Omega})$. The \term{(geometric) pushforward} $\theta_* T$ (often denoted by \enquote{$\theta_\# T$} in the literature, but this can lead to confusion with the measure-theoretic pushforward, cf.~\cite[p.~32]{AmbrosioFuscoPallara00book}) is defined via
\[
  \dprb{\theta_* T,\omega} := \dprb{T, \theta^*\omega}, \qquad \omega \in \Dcal^k(\R^m),
\]
where $\theta^*\omega$ denotes the pullback of the $k$-form $\omega$. If $\theta$ is only Lipschitz continuous, then $\theta_* T$ is defined via the homotopy formula and a smoothing argument, see~\cite[Lemma~7.4.3]{KrantzParks08book}. It holds that $\theta_* T \in \Irm_k(\cl{\theta(\Omega)})$, see, for instance,~\cite[(3) on p.~197]{KrantzParks08book}. Denoting the approximate derivative of $\theta$ (which is defined almost everywhere) with respect to the $\Hcal^k$-rectifiable set $R$ by $D^R \theta$ (i.e., $D^R \theta(x)$ is the restriction of $D\theta(x)$ to $\Trm_x R$), we have
\begin{equation} \label{eq:pushforward}
  \dprb{\theta_* T, \omega} = \int \dprb{D^R\theta(\vec{T}(x)), \omega(\theta(x))} \dd \tv{T}(x), \qquad \omega \in \Dcal^k(\R^m).
\end{equation}
We note further that
\begin{equation} \label{eq:pushforward_bdry}
  \partial (\theta_* T) = \theta_*(\partial T).
\end{equation}

As convergence for integral currents we use the weak* convergence, i.e., we say that a sequence $(T_j) \subset \Irm_k(\R^d)$ \term{converges weakly*} to $T \in \Dcal_k(\R^d)$, in symbols \enquote{$T_j \toweakstar T$}, if
\[
  \dprb{T_j,\omega} \to \dprb{T,\omega} \qquad\text{for all $\omega \in \Dcal^k(\R^d)$.}
\]
Moreover, for $T \in \Irm_k(\R^d)$, the \term{(global, integral) Whitney flat norm} is given by
\[
  \Fbf(T) := \inf \, \setB{ \Mbf(Q) + \Mbf(R) }{ \text{$Q \in \Irm_{k+1}(\R^d)$, $R \in \Irm_k(\R^d)$ with $T = \partial Q + R$} }.
\]
Then, one can consider the \term{flat convergence} $\Fbf(T_j-T) \to 0$ for a sequence $(T_j) \subset \Irm_k(\R^d)$ as above. We quote two central results on the weak* convergence of integral $k$-currents:

First, in a bounded Lipschitz domain $\Omega \subset \R^d$, the weak* convergence is actually equivalent to the flat convergence under a uniform mass bound, see~\cite[Theorem~8.2.1]{KrantzParks08book}:

\begin{proposition} \label{prop:flat}
Let $(T_j) \subset \Irm_k(\cl{\Omega})$ with 
\[
  \sup_{j\in\N} \, \bigl( \Mbf(T_j) + \Mbf(\partial T_j) \bigr) < \infty.
\]
Then, $T_j \toweakstar T$ for some $T \in \Irm_k(\cl{\Omega})$ if and only if $\Fbf(T_j-T) \to 0$.
\end{proposition}

Second, compactness for integral currents is usually established via the Federer--Fleming compactness theorem, see~\cite[4.2.17]{Federer69book} or~\cite[Theorems~7.5.2,~8.2.1]{KrantzParks08book}:

\begin{proposition} \label{prop:FF}
Let $(T_j) \subset \Irm_k(\cl{\Omega})$ with 
\[
  \sup_{j\in\N} \, \bigl( \Mbf(T_j) + \Mbf(\partial T_j) \bigr) < \infty.
\]
Then, there is a subsequence (not relabeled) of $(T_j)$ and a $T \in \Irm_k(\cl{\Omega})$ such that $\Fbf(T_j-T) \to 0$ or, equivalently, $T_j \toweakstar T$. Moreover,
\begin{align*}
  \Mbf(T) &\leq \liminf_{j\to\infty} \, \Mbf(T_j), \\
  \Mbf(\partial T) &\leq \liminf_{j\to\infty} \, \Mbf(\partial T_j).
\end{align*}
\end{proposition}

\subsection{Slicing of integral currents} \label{sc:slice}

An integral current $S = m \, \vec{S} \, \Hcal^{k+1} \restrict R \in \Irm_{k+1}(\R^n)$ can be \enquote{sliced} with respect to the level sets of a Lipschitz map $f \colon \R^n \to \R$, see~\cite[Section~7.6]{KrantzParks08book} or~\cite[Section~4.3]{Federer69book}. For $\Lcal^1$-almost every $t \in \R$, the following statements hold:
\begin{enumerate}[(i)]
  \item The set $R|_t := f^{-1}(\{t\}) \cap R$ is (countably) $\Hcal^k$-rectifiable.
  \item For $\Hcal^{k+1}$-almost every $z \in R$, the approximate tangent spaces $\Trm_z R$ and $\Trm_z R|_t$, as well as the approximate gradient $\nabla^R f(z)$, i.e., the projection of $\nabla f(z)$ onto $\Trm_z R$, exist and
\[
  \Trm_z R = \spn \bigl\{ \Trm_z R|_t, \xi(z) \bigr\},  \qquad
  \xi(z) := \frac{\nabla^R f(z)}{\abs{\nabla^R f(z)}} \perp \Trm_z R|_t.
\]
  Moreover, $\xi(z)$ is simple and has unit length.
  \item With
\[
  m_+(z) := \begin{cases}
    m(z) &\text{if $\nabla^R f(z) \neq 0$,} \\
    0    &\text{otherwise,}  
  \end{cases}
  \qquad\qquad
  \xi^*(z) := \frac{D^R f(z)}{\abs{D^R f(z)}} \in \Wedge^1 \R^n,
\]
  where $D^R f(z)$ is the restriction of the differential $Df(z)$ to $\Trm_z R$, and
\[
  \vec{S}|_t(z) := \vec{S}(z) \restrict \xi^*(z) 
  \in \Wedge_k \Trm_z R|_t
  \subset \Wedge_k \Trm_z R,
\]
the \term{slice}
\[
  S|_t := m_+ \, \vec{S}|_t \, \Hcal^k \restrict R|_t
\]
  is an integral $k$-current, $S|_t \in \Irm_k(\R^n)$.
  \item The coarea formula for slices
\[
  \int_R g \, \abs{\nabla^R f} \dd \Hcal^{1+k} = \int \int_{R|_t} g \dd \Hcal^k \dd t
\]
  holds for all $g \colon R \to \R^N$ that are $\Hcal^{k+1}$-measurable and such that $g \, \abs{\nabla^R f}$ is $\Hcal^{1+k}$-integrable on $R$ or $g \geq 0$. In particular,
\[ \qquad
  \int_R \abs{\nabla^R f} \dd \tv{S} = \int \Mbf(S|_t) \dd t.
\]
  \item The cylinder formula
\begin{equation} \label{eq:cylinder}
  S|_t = \partial(S \restrict \{f<t\}) - (\partial S) \restrict \{f<t\}
\end{equation}
  and the boundary formula
\[
  \partial (S|_t) = - (\partial S)|_t
\]
  hold.
\end{enumerate}

\subsection{Approximation of integral currents}

Finally, we recall the following approximation result, which is proved in greater generality in~\cite[Theorem~1.2]{ChambolleFerrariMerlet21} based on~\cite{ColomboDeRosaMarcheseStuvard17}. To state it, we let $\mathrm{IP}_k(\cl{\Omega})$ be the set of $k$-dimensional \term{integral polyhedral chains} with support in $\cl{\Omega}$, that is, those $P \in \Irm_k(\cl{\Omega})$ that can be written in the form
\[
  P = \sum_{\ell = 1}^N p_\ell \dbr{\sigma_\ell},
\]
where the $\sigma_\ell$ are oriented convex $k$-polytopes ($\ell \in \{1,\ldots,N\}$), $\dbr{\sigma_\ell}$ denotes the integral $k$-current associated with $\sigma_\ell$ (with unit multiplicity), and $p_\ell \in \N$.

\begin{proposition} \label{prop:approx}
Let $T \in \Irm_k(\cl{\Omega})$ with $\partial T \in \mathrm{IP}_{k-1}(\cl{\Omega})$. Then, for every $\eps > 0$, there is $P \in \mathrm{IP}_k(\cl{\Omega'})$, where $\Omega' := \Omega + B(0,\eps)$, and $Q \in \Irm_{k+1}(\cl{\Omega'})$ such that
\[
  T = \partial Q + P
\]
with
\[
  \Mbf(Q) < \eps,  \qquad
  \Mbf(P) < \Mbf(T) + \eps.
\]
\end{proposition}

The important point here is that, unlike in the classical deformation theorem (see~\cite[4.2.9]{Federer69book} or~\cite[Section~7.7]{KrantzParks08book}), the mass estimate $\Mbf(P) < \Mbf(T) + \eps$ holds. We also refer to~\cite[Theorem~8.22]{FedererFleming60} for an earlier result in this direction; note that Proposition~\ref{prop:approx} is a deformation result (the difference $T-P$ is expressed as a boundary) and not merely an approximation theorem like~\cite[Corollary~8.23]{FedererFleming60} or~\cite[Theorem~4.2.24]{Federer69book}.

The proof is essentially contained in~\cite[Theorem~1.2]{ChambolleFerrariMerlet21} and~\cite[Proposition~2.7]{ColomboDeRosaMarcheseStuvard17} (for us, the easier argument of Section~1.2 in~\cite{ChambolleFerrariMerlet21} suffices). An inspection of this proof yields that if $T$ is an \emph{integral} current, then also $P$ is an \emph{integral} polyhedral chain ($p_\ell \in \N$). Moreover, also $Q$ (whose boundary is the difference between $T$ and $P$) is integral since it is constructed via the homotopy formula and the deformation theorem, both of which yield integral currents in the present situation.

\section{BV-theory of integral currents} \label{sc:BVcurr}

In~\cite{HudsonRindler21?,Rindler21b?} (also see~\cite{Kampschulte17PhD}) it is discussed at some length why it is beneficial to consider evolutions of integral $k$-currents to be identified with the space-time $(1+k)$-current \enquote{traced out} by the moving $k$-current. Below we will introduce the \enquote{variation} of this space-time current as the total traversed spatial area, but not letting opposite movements cancel each other. One noteworthy feature of the present theory is that at jump times the space-time currents contain also a notion of \enquote{jump transient} in their vertical pieces. In this sense, our theory is closer to the BV-liftings investigated in~\cite{JungJerrard04,RindlerShaw19} or cartesian currents~\cite{GiaquintaModicaSoucek98book1,GiaquintaModicaSoucek98book2} than to the classical theory of BV-maps~\cite{AmbrosioFuscoPallara00book}.

\subsection{Variation of space-time integral currents} \label{sc:Var}

In the following we will often work in the (Galilean) space-time $\R^{1+d} \cong \R \times \R^d$, where the first component takes the role of \enquote{time} and the remaining components take the role of \enquote{space}. The unit vectors in $\R^{1+d}$ are denoted by $\ee_0, \ee_1,\ldots,\ee_d$ with $\ee_0$ the \enquote{time} unit vector (pointing in the positive direction). It will be convenient to write the orthogonal projection onto the time component as $\tbf \colon \R^{1+d} \to \R \times \{0\}^d \cong \R$, $\tbf(t,x) := t$, and the orthogonal projection onto the space component as $\pbf \colon \R^{1+d} \to \{0\} \times \R^d \cong \R^d$, $\pbf(t,x) := x$. We also denote the linear extensions of these projections to multi-vectors by the same symbols.

Let $S \in \Irm_{1+k}([\sigma,\tau] \times \cl{\Omega})$, where $\sigma < \tau$. We define the \term{(space-time) variation} and \term{(space-time) boundary variation} of $S$ in the interval $I \subset [\sigma,\tau]$ via, respectively,
\begin{align} 
  \Var(S;I) &:= \int_{I \times \R^d} \abs{\pbf(\vec{S})} \dd \tv{S},  \label{eq:VarM} \\
  \Var(\partial S;I) &:= \int_{I \times \R^d} \abs{\pbf(\overrightarrow{\partial S})} \dd \tv{\partial S}.  \label{eq:VarM_bdry}
\end{align}
If $[\sigma,\tau] = [0,1]$, then we also write $\Var(S)$, $\Var(\partial S)$ for $\Var(S;[0,1])$, $\Var(\partial S;[0,1])$. Clearly, the variation is additive in the interval $I$, that is, for $\sigma \leq r < s < t \leq \tau$ it holds that
\begin{align*}
  \Var(S;[r,t)) &= \Var(S;[r,s)) + \Var(S;[s,t)), \\
  \Var(S;(r,t]) &= \Var(S;(r,s]) + \Var(S;(s,t]).
\end{align*}
Since $\abs{\pbf(\vec{S})} \leq 1$,
\begin{equation} \label{eq:VarM_est}
  \Var(S;I) \leq \Mbf(S \restrict (I \times \R^d)) \leq \Mbf(S)
\end{equation}
and likewise for the boundary variation. A reverse estimate will be given in Lemma~\ref{lem:mass} below.

Via the slicing theory of currents, for $\Lcal^1$-almost every $t \in [\sigma,\tau]$ we can define
\[
  S(t) := \pbf_*(S|_t) \in \Irm_k(\cl{\Omega}),
\]
where $S|_t \in \Irm_k([\sigma,\tau] \times \cl{\Omega})$ denotes the slice of $S$ with respect to time (i.e., with respect to $\tbf$). Note that if $S \in \Irm_{1+k}([\sigma,\tau] \times \cl{\Omega})$ has a \term{jump} at $t \in [\sigma,\tau]$, that is, $\tv{S}(\{t\} \times \R^d) > 0$, then $S|_t$ does not exist and the vertical piece $S \restrict (\{t\} \times \R^d)$ takes the role of a \enquote{jump transient}, i.e., the specific surface connecting the endpoints of the jump.

We also introduce the set of \term{integral $(1+k)$-currents with Lipschitz continuity}, or \term{Lip-integral $(1+k)$-currents}, as follows:
\begin{align*}
  \Irm^\Lip_{1+k}([\sigma,\tau] \times \cl{\Omega})
  := \setBB{ S \in \Irm_{1+k}([\sigma,\tau] \times \cl{\Omega}) }{ &\esssup_{t \in [\sigma,\tau]} \, \bigl( \Mbf(S(t)) + \Mbf(\partial S(t)) \bigr) < \infty, \\
  & \tv{S}(\{\sigma,\tau\} \times \R^d) = 0, \\
  &t \mapsto \Var(S;[\sigma,t]) \in \Lip([\sigma,\tau]), \\
  &t \mapsto \Var(\partial S;(\sigma,t)) \in \Lip([\sigma,\tau]) },
\end{align*}
where $\Lip([\sigma,\tau])$ denotes the space of scalar Lipschitz functions on the interval $[\sigma,\tau]$. We remark that there seems to be little point in defining a space like \enquote{$\BV([\sigma,\tau];\Irm_k(\cl{\Omega}))$} since by~\eqref{eq:VarM_est} and Lemma~\ref{lem:mass} below the mass and variation are comparable in the presence of a uniform mass bound on the slices.

Let us consider some examples to illustrate the above notions.

\begin{example} \label{ex:BV}
Let $u \in \BV([0,1])$ (see~\cite{AmbrosioFuscoPallara00book}) and define $S_u := \tau \, \Hcal^1 \restrict \graph(u)$ with
\[
  \graph(u) := \setb{ (t,u^\theta(t)) }{ t \in [0,1], \; \theta \in [0,1] },
\]
where $u^\theta(t) := (1-\theta)u^-(t) + \theta u^+(t)$ is the affine jump between the left and right limits $u^\pm(t) = u(t\pm)$ (which are equal to $u(t)$ if $t$ is a continuity point), and $\tau$ is the forward-pointing unit tangent to $\graph(u)$ (with $\tau \cdot \ee_0 \geq 0$). In this case, $\Var(S_u;I) = \Var(u;I) = \abs{Du}(I)$ for every interval $I \subset [0,1]$. This can be seen as follows: By a smoothing argument and Reshetnyak's continuity theorem (see, e.g.,~\cite[Theorem~2.39]{AmbrosioFuscoPallara00book}) we may without loss of generality assume that $u \in \Crm^1([0,1])$. Then,
\[
  \Var(u;I)
  = \int_I \abs{\dot{u}} \dd t
  = \int_{\graph(u) \cap (I \times \R)} \frac{\abs{\dot{u}}}{\sqrt{1+\abs{\dot{u}}^2}} \dd \Hcal^1 
  = \int_{\graph(u) \cap (I \times \R)} \abs{\pbf(\tau)} \dd \Hcal^1
  = \Var(S_u;I),
\]
where we used the area formula, Proposition~\ref{prop:area}. Clearly, $S_u \in \Irm^\Lip_{1+0}([0,1] \times \R)$ if and only if $u$ is Lipschitz. In this sense, the classical notions of BV- and Lipschitz-functions (with scalar values) constitute the $0$-dimensional case of our theory.
\end{example}

\begin{example} \label{ex:defo_cone}
Let $\Omega$ be star-shaped with vertex $p \in \Omega$, and let $T \in \Irm_k(\cl{\Omega})$. Define $H(t,x) := (1-t)p + tx$ and $\bar{H}(t,x) := (t,H(t,x))$. The \term{cone}
\[
  p \cone T := \bar{H}_* (\dbr{(0,1)} \times T) \in \Irm^\Lip_{1+k}([0,1] \times \cl{\Omega})
\]
satisfies $\partial(p \cone T) = \delta_1 \times T - p \cone \partial T$ (see~\eqref{eq:product_bdry},~\eqref{eq:pushforward_bdry}).
\end{example}

\begin{example}\label{ex:defo_Cantor}
Let $f_C \colon [0,1] \to [0,1]$ be the Cantor--Vitali function~\cite[Example~1.67]{AmbrosioFuscoPallara00book} and let $S_C$ be the \enquote{Cantor cone}, that is, the set in $\R^{1+2}$ obtained by rotating the graph of $f_C$ around the time axis. As the graph of $f_C$ is $1$-rectifiable (with length $2$), we get that $S_C$ is $2$-rectifiable. Hence, with a choice of orientation, $S_C \in \Irm_2(\R^3)$. Then, $S_C(t) \in \Irm_1(\R^2)$ for $t \in [0,1]$ is the circle lying around the origin with radius $f_C(t)$ and $\Var(S_C;[0,t]) = \pi f_C(t)^2$. Hence, $S_C \notin \Irm^\Lip_{1+1}([0,T] \times \R^2)$.
\end{example}

Like the classical variation, also our space-time variation is invariant with respect to time rescalings:

\begin{lemma} \label{lem:rescale}
Let $S \in \Irm_{1+k}([\sigma,\tau] \times \cl{\Omega})$ and let $a \in \Lip([\sigma,\tau])$ be injective. Then,
\[
  a_* S := [(t,x) \mapsto (a(t),x)]_* S \in \Irm_{1+k}(a([\sigma,\tau]) \times \cl{\Omega})
\]
with
\[
  (a_* S)(a(t)) = S(t),  \qquad t \in [\sigma,\tau],
\]
and
\begin{align*}
  \Var(a_* S;a([\sigma,\tau])) &= \Var(S;[\sigma,\tau]), \\
  \Var(\partial(a_* S);a([\sigma,\tau])) &= \Var(\partial S;[\sigma,\tau]), \\
  \esssup_{t \in a([\sigma,\tau])} \, \Mbf((a_* S)(t)) &= \esssup_{t \in [\sigma,\tau]} \, \Mbf(S(t)), \\
  \esssup_{t \in a([\sigma,\tau])} \, \Mbf(\partial(a_* S)(t)) &= \esssup_{t \in [\sigma,\tau]} \, \Mbf(\partial S(t)).
\end{align*}
If $S \in \Irm^\Lip_{1+k}([\sigma,\tau] \times \cl{\Omega})$, then also $a_* S \in \Irm^\Lip_{1+k}(a([\sigma,\tau]) \times \cl{\Omega})$.
\end{lemma}

\begin{proof}
If $S = m \, \vec{S} \, \Hcal^{1+k} \restrict R$ with a countably $(1+k)$-rectifiable set $R \subset \R^{1+d}$, we get (see, e.g.,~\cite[(3) on p.~197]{KrantzParks08book})
\[
  a_* S = m \circ a^{-1} \, \frac{(D^R a \circ a^{-1})[\vec{S} \circ a^{-1}]}{\abs{(D^R a \circ a^{-1})[\vec{S} \circ a^{-1}]}} \, \Hcal^k \restrict a(R),
\]
where here and in the following we identify $a$ with the space-time map $(t,x) \mapsto (a(t),x)$. Since $a$ only transforms the time coordinate,
\[
  \pbf((D^R a \circ a^{-1})[\vec{S} \circ a^{-1}]) = \pbf((Da \circ a^{-1})[\vec{S} \circ a^{-1}]) = \pbf(\vec{S} \circ a^{-1})
\]
and
\[
  \Jbf^R_k a = \abs{(D^R a)[\vec{S}]}.
\]
Hence,
\begin{align*}
  \Var(a_* S;a([\sigma,\tau]))
  &= \int_{a(R)} \absBB{\pbf \biggl( \frac{(D^R a \circ a^{-1})[\vec{S} \circ a^{-1}]}{\abs{(D^R a \circ a^{-1})[\vec{S} \circ a^{-1}]}} \biggr)} \; m \circ a^{-1} \dd \Hcal^k \\
  &= \int_{a(R)} \frac{\abs{\pbf(\vec{S} \circ a^{-1})}}{\abs{\Jbf^R_k a \circ a^{-1}}} \; m \circ a^{-1} \dd \Hcal^k \\
  &= \int_{R} \abs{\pbf(\vec{S})} \; m \dd \Hcal^k \\
  &= \Var(S;[\sigma,\tau]),
\end{align*}
where we used the area formula (Proposition~\ref{prop:area}). The equality for the boundary variation follows in the same way. The additional claim about Lip-integral currents is then also clear.
\end{proof}

\subsection{Pointwise variation and mass estimates} \label{sc:pVarF}

We now explore how our definition of variation relates to the variation with respect to the flat norm.

Let $\sigma < s < t < \tau$ such that $S(s),S(t)$ are defined for $S \in \Irm_{1+k}([\sigma,\tau] \times \R^d)$ (in particular, $\tv{S}(\{s\} \times \R^d) = \tv{S}(\{t\} \times \R^d) = 0$). Then, for
\[
  Q := \pbf_* [S \restrict ([s,t] \times \R^d)] \in \Irm_{k+1}(\cl{\Omega})
\]
it holds that (see the cylinder formula~\eqref{eq:cylinder})
\[
  \partial Q = \pbf_* \bigl[ \partial ( S \restrict ([s,t] \times \R^d)) \bigr]
  = S(t) - S(s) + \pbf_* \bigl[ (\partial S) \restrict ([s,t] \times \R^d) \bigr].
\]
Next, we observe
\begin{equation}  \label{eq:FVar_est}
  \Mbf(Q) = \sup_{\substack{\omega \in \Dcal^{1+k}(\R^d)\\\abs{\omega} \leq 1}} \int_{[s,t] \times \R^d} \dprb{\pbf(\vec{S}(t,x)), \omega(x)} \dd \tv{S}(t,x) 
  \leq \Var(S;[s,t]). 
\end{equation}
Also setting $R := -\pbf_* [(\partial S) \restrict ([s,t] \times \R^d)] \in \Irm_k(\cl{\Omega})$, we have $S(t) - S(s) = \partial Q + R$ and thus
\[
  \Fbf(S(t)-S(s)) \leq \Mbf(Q) + \Mbf(R)
  \leq \Var(S;[s,t]) + \Var(\partial S;[s,t])
\]
since $Q,R$ are admissible in the definition of $\Fbf(S(t)-S(s))$. From this we immediately obtain for the \term{pointwise $\Fbf$-variation}
\[
  \FpV(S;[s,t]) := \sup \setBB{ \sum_{\ell = 1}^N \Fbf(S(t_{\ell-1})-S(t_\ell)) }{ \sigma = t_0 < t_1 < \cdots t_N = \tau, \; \text{$S(t_\ell)$ defined} }
\]
the estimates
\begin{align*}
  \FpV(S;[s,t]) &\leq \Var(S;[s,t]) + \Var(\partial S;[s,t]), \\
  \FpV(\partial S;[s,t]) &\leq \Var(\partial S;[s,t]).
\end{align*}
Thus, $t \mapsto S(t)$ and $t \mapsto \partial S(t)$ are functions of bounded (pointwise) variation with respect to $\Fbf$.

Assume now additionally a uniform bound on $\Mbf(S(t)) + \Mbf(\partial S(t))$ for $t \in [\sigma,\tau]$. One then obtains, using the Federer--Fleming compactness theorem, Proposition~\ref{prop:FF}, that at every $t \in [\sigma,\tau]$ the left and right limits exist with respect to weak* convergence in $\Irm_k(\cl{\Omega})$ (only one-sided limits at $\sigma,\tau$). Indeed, for instance, if there were sequences $t_j \todown t$ and $\tilde{t}_j \todown t$ with $0 < \delta < \Fbf(S(t_j)-S(\tilde{t}_j))$ for all $j$, then, up to selecting a subsequence, $\infty = \sum_j \Fbf(S(t_j)-S(\tilde{t}_j)) \leq \FpV(S;[t,\tau]) < \infty$, which is a contradiction; likewise for left limits.

Thus, we may define the right-continuous \term{good representative} $\tilde{S} \colon [\sigma,\tau) \to \Irm_k(\cl{\Omega})$ of $S$ for any $t \in [\sigma,\tau)$ as
\[
  \tilde{S}(t) := S(t+) = \wslim_{s \todown t} S(s)  \quad\text{in $\Irm_k(\cl{\Omega})$,}
\]
which satisfies $\tilde{S}(t) = S(t)$ for $\Lcal^1$-almost every $t \in [\sigma,\tau)$. In the following we will drop the tilde and just refer to $\tilde{S}(t)$ as $S(t)$.

From the above arguments we further obtain the following $\Fbf$-Poincar\'{e} inequality for the good representative:
\begin{equation} \label{eq:PoincareF}
  \Fbf(S(s) - S(t)) \leq \FpV(S;[s,t]) \leq \Var(S;[s,t]) + \Var(\partial S;[s,t]),  \qquad s,t \in [\sigma,\tau].
\end{equation}
Here we have additionally set $S(\tau) := S(\tau-)$. This implies in particular that if $t \in [\sigma,\tau]$ is a continuity point of $t \mapsto \Var(S;[\sigma,t])$ and $t \mapsto \Var(\partial S;[\sigma,t])$, then $t$ is also a (weak*) continuity point of $t \mapsto S(t)$, that is, $S(t-) = S(t+)$. Note, however, that the inequality~\eqref{eq:PoincareF} is too weak to give a uniform mass bound on $t \mapsto S(t)$ in terms of the variation.

If even $S \in \Irm^\Lip_{1+k}([\sigma,\tau] \times \cl{\Omega})$, then the $\Fbf$-Lipschitz constant
\[
  L := \sup_{s,t \in [\sigma,\tau]} \frac{\Fbf(S(s)-S(t))}{\abs{s-t}}
\]
of (the good representative of) $S$ is finite and $t \mapsto S(t)$ is continuous with respect to the weak* convergence in $\Irm_k(\cl{\Omega})$. Moreover,
\begin{equation} \label{eq:partialS_trace}
  \partial S \restrict (\{\sigma,\tau\} \times \R^d) = \delta_\tau \times S(\tau-) - \delta_\sigma \times S(\sigma+),
\end{equation}
which can be seen by considering $\tilde{S} := S - \dbr{(-\infty,\sigma)} \times \pbf_*(\partial S \restrict (\{\sigma\} \times \R^d)) + \dbr{(\tau,\infty)} \times \pbf_*(\partial S \restrict (\{\tau\} \times \R^d))$ and using that $\tv{S}(\{\sigma,\tau\} \times \R^d) = 0$ to see that $t \mapsto \tilde{S}(t)$ has $\sigma,\tau$ as (weak*) continuity points. In conclusion, $S(\sigma+), S(\tau-)$ can be considered the left and right \term{trace values} of $S$.

It is important to notice that, in general, $\FpV(S;[\sigma,\tau])$ is strictly smaller than $\Var(S;[\sigma,\tau]) + \Var(\partial S;[\sigma,\tau])$ since $\FpV(S;[\sigma,\tau])$ always counts the jump variations via the $\Fbf$-distance between the jump endpoints $S(t\pm)$. On the other hand, as we have mentioned already, a $(1+k)$-surface in $\Irm_{1+k}([\sigma,\tau] \times \cl{\Omega})$ always additionally specifies the \emph{jump transients}, which may not be $\Fbf$-minimal.

The next \enquote{Pythagoras} lemma gives an estimate for the mass of an integral $(1+k)$-current in terms of the masses of the slices and the variation.

\begin{lemma} \label{lem:mass}
Let $S = m \, \vec{S} \, \Hcal^{1+k} \restrict R \in \Irm_{1+k}([\sigma,\tau] \times \cl{\Omega})$. Then,
\begin{equation} \label{eq:decomp}
  \abs{\nabla^R \tbf}^2 + \abs{\pbf(\vec{S})}^2 = 1  \qquad \text{$\tv{S}$-a.e.}
\end{equation}
and
\begin{align*}
  \Mbf(S) &\leq \int_\sigma^\tau \Mbf(S(t)) \dd t + \Var(S;[\sigma,\tau]) \\
  &\leq \abs{\sigma-\tau} \cdot \esssup_{t \in [\sigma,\tau]} \, \Mbf(S(t)) + \Var(S;[\sigma,\tau]).
\end{align*}
\end{lemma}

\begin{proof}
Let us first recall that for $\Lcal^1$-almost every $t \in [0,T]$ and $\tv{S|_t}$-almost every $(t,x)$ the approximate tangent spaces $\Trm_{(t,x)} R, \Trm_{(t,x)} R|_t$ as well as the approximate differential $D^R \tbf(t,x)$ and the approximate gradient $\nabla^R \tbf(t,x)$ exist (see Section~\ref{sc:slice}). Moreover,
\[
  \Trm_{(t,x)} R = \spn \bigl\{ \Trm_{(t,x)} R|_t, \xi(t,x) \bigr\},  \qquad
  \xi(t,x) := \frac{\nabla^R \tbf(t,x)}{\abs{\nabla^R \tbf(t,x)}} \perp \Trm_{(t,x)} R|_t.
\]
Thus, with
\[
  \xi^*(x) := \frac{D^R \tbf(t,x)}{\abs{D^R \tbf(t,x)}} \in \Wedge^1 \R^{1+d},
\]
we have
\[
  \vec{S}|_t = \vec{S} \restrict \xi^*,  \qquad
  \vec{S} = \xi \wedge \vec{S}|_t.
\]
The second equality here follows from the general relation $\xi \wedge (\tau\restrict \xi^*) = \tau$ for any $\tau \in \Wedge_{1+k} \R^{1+d}$ with $\xi \wedge \tau = 0$ (see~\cite[1.5.3]{Federer69book}).

In the following we fix $t,x$ as above and suppress the arguments $(t,x)$. We observe that (note $\abs{\xi} = 1$)
\[
  \abs{\nabla^R \tbf}
  = \abs{(\ee_0 \cdot \xi)\xi}
  = \abs{\xi \cdot \ee_0}
  = \abs{(\xi \cdot \ee_0)\ee_0}
  = \abs{\tbf(\xi)}
\]
and (recall $\pbf(\xi) \perp \vec{S}|_t$)
\[
  \abs{\pbf(\vec{S})}
  = \abs{\pbf(\xi) \wedge \vec{S}|_t}
  = \abs{\pbf(\xi)}.
\]
Since $\abs{\tbf(\xi)}^2 + \abs{\pbf(\xi)}^2 = 1$, we obtain~\eqref{eq:decomp}. Then, 
\begin{align*}
  \Mbf(S) &= \int_R \sqrt{\abs{\nabla^R \tbf}^2 + \abs{\pbf(\vec{S})}^2}  \dd \tv{S} \\
  &\leq \int_R \abs{\nabla^R \tbf} + \abs{\pbf(\vec{S})}  \dd \tv{S} \\
  &=\int_\sigma^\tau \Mbf(S|_t) \dd t + \Var(S;[\sigma,\tau]),
\end{align*}
where in the last line we have used the coarea formula for slices and the definition of the variation. This yields the second claim.
\end{proof}

\begin{example} \label{ex:MLinfty_necessary}
Let $R \in \Irm_k(\R^d)$ be an integral $k$-current, $k \geq 1$, with $\Mbf(R) = 1$, but $\Mbf(\partial R) = N \in \N$ (e.g., a disk with \enquote{rough} boundary). Then, if $p$ lies in the relative interior of $R$, we define (recalling the definition of the cone in Example~\ref{ex:defo_cone})
\[
  S := - \partial (p \cone R) = p \cone \partial R - \delta_1 \times R. 
\]
One computes that $\Var(S;[0,1]) = 2$ (the cone and the endpoint cap each have variation $\Mbf(R) = 1$) and $\Var(\partial S;[0,1]) = 0$. On the other hand, $\Mbf(S(t)) = tN$ for almost every $t \in [0,1]$. This shows that $\esssup_{t \in [0,1]} \Mbf(S(t))$ is not controlled by any expression involving only $\Var(S;[0,1])$ and $\Var(\partial S;[0,1])$ besides constants, unless $k = 0$ and we are in the case of BV-maps and every slice is a Dirac point mass (see Example~\ref{ex:BV}).
\end{example}

\subsection{Weak* convergence and compactness}

We say that $(S_j) \subset \Irm_{1+k}([\sigma,\tau] \times \cl{\Omega})$ \term{converges BV-weakly*} to $S \in \Irm_{1+k}([\sigma,\tau] \times \cl{\Omega})$ as $j \to \infty$, in symbols \enquote{$S_j \toweakstar S$ in BV}, if
\[
  \left\{ \begin{aligned}
    S_j &\toweakstar S        &&\text{in $\Irm_{1+k}([\sigma,\tau] \times \cl{\Omega})$,} \\
    S_j(t) &\toweakstar S(t)  &&\text{in $\Irm_k(\cl{\Omega})$ for $\Lcal^1$-almost every $t \in [\sigma,\tau]$.}
  \end{aligned} \right.
\]
For this convergence we have the following version of Helly's selection principle:

\begin{theorem}[Compactness theorem] \label{thm:current_Helly}
Assume that the sequence $(S_j) \subset \Irm_{1+k}([\sigma,\tau] \times \cl{\Omega})$ satisfies
\[
  \esssup_{t \in [\sigma,\tau]} \, \bigl( \Mbf(S_j(t)) + \Mbf(\partial S_j(t)) \bigr) + \Var(S_j;[\sigma,\tau]) + \Var(\partial S_j;[\sigma,\tau]) \leq C < \infty
\]
for all $j \in \N$. Then, there exists $S \in \Irm_{1+k}([\sigma,\tau] \times \cl{\Omega})$ and a (not relabelled) subsequence such that
\[
  S_j \toweakstar S  \quad\text{in BV.}
\]
Moreover,
\begin{align*}
  \esssup_{t \in [\sigma,\tau]} \, \Mbf(S(t)) &\leq \liminf_{j \to \infty} \, \; \esssup_{t \in [\sigma,\tau]} \, \Mbf(S_j(t)), \\
  \esssup_{t \in [\sigma,\tau]} \, \Mbf(\partial S(t)) &\leq \liminf_{j \to \infty} \, \; \esssup_{t \in [\sigma,\tau]} \, \Mbf(\partial S_j(t)), \\
  \Var(S;[\sigma,\tau]) &\leq \liminf_{j \to \infty} \, \Var(S_j;[\sigma,\tau]), \\
  \Var(\partial S;(\sigma,\tau)) &\leq \liminf_{j \to \infty} \, \Var(\partial S_j;(\sigma,\tau)).
\end{align*}
If additionally $(S_j) \subset \Irm^\Lip_{1+k}([\sigma,\tau] \times \cl{\Omega})$ such that the Lipschitz constants $L_j$ of the scalar maps $t \mapsto \Var(S_j;[\sigma,t]) + \Var(\partial S_j;(\sigma,t))$ are uniformly bounded, then also $S \in \Irm^\Lip_{1+k}([\sigma,\tau] \times \cl{\Omega})$ with Lipschitz constant bounded by $\liminf_{j\to\infty} L_j$. Moreover, in this case, $S_j(t) \toweakstar S(t)$ in $\Irm_k(\cl{\Omega})$ for \emph{every} $t \in [\sigma,\tau)$.
\end{theorem}

\begin{proof}
From the assumptions we infer a uniform bound on the masses $\Mbf(S_j)$ and $\Mbf(\partial S_j)$ via Lemma~\ref{lem:mass}. Then, the first convergence $S_j \toweakstar S$ in $\Irm_{1+k}([\sigma,\tau] \times \cl{\Omega})$, up to selecting a subsequence, follows directly from the Federer--Fleming compactness theorem in $\Irm_{1+k}([\sigma,\tau] \times \cl{\Omega})$, see Proposition~\ref{prop:FF}.

By the cylinder formula~\eqref{eq:cylinder},
\[
  S|_t = \partial(S \restrict \{\tbf<t\}) - (\partial S) \restrict \{\tbf<t\}
\]
and likewise for $S_j|_t$. If $\tv{S_j} + \tv{\partial S_j} \toweakstar \nu$ in $\Mcal^+([\sigma,\tau] \times \R^d)$ (for a subsequence), then standard results in measure theory (see, e.g.,~\cite[Theorem~1.62~(b)]{AmbrosioFuscoPallara00book}) imply that for all $t \in [\sigma,\tau]$ with $\nu(\{t\} \times \R^d) = 0$ it holds that
\[
  \dprb{S_j,\ONE_{\{\tbf<t\}} \wedge d\omega} \to \dprb{S,\ONE_{\{\tbf<t\}} \wedge d\omega},  \qquad
  \dprb{\partial S_j,\ONE_{\{\tbf<t\}} \wedge \omega} \to \dprb{\partial S,\ONE_{\{\tbf<t\}} \wedge \omega}.
\]
for all $\omega \in \Dcal^k(\R^{1+d})$. Thus, for these $t$,
\[
  \partial \bigl( S_j \restrict \{\tbf<t\} \bigr) \toweakstar \partial \bigl( S \restrict \{\tbf<t\} \bigr), \qquad
  \partial S_j \restrict \{\tbf<t\} \toweakstar \partial S \restrict \{\tbf<t\}.
\]
Since there are only at most countably many $t$'s with $\nu(\{t\} \times \R^d) > 0$, we obtain that 
\[
  S_j|_t \toweakstar S|_t  \qquad\text{for $\Lcal^1$-almost every $t$.}
\]
This shows the second convergence $S_j(t) \toweakstar S(t)$ in $\Irm_k(\cl{\Omega})$ for $\Lcal^1$-almost every $t \in [\sigma,\tau]$. 

The lower semicontinuity of the mass and variation follow in the usual way from the weak* convergences. Indeed, the variation $\Var(S_j;[\sigma,\tau])$ is lower semicontinuous by Reshetnyak's lower semicontinuity theorem (see, for instance,~\cite[Theorem~2.38]{AmbrosioFuscoPallara00book}), i.e.,
\begin{align*}
  \Var(S;[\sigma,\tau]) &= \int_{[\sigma,\tau] \times \Omega} \abs{\pbf(\vec{S})} \dd \tv{S} \\
  &\leq \liminf_{j \to \infty} \int_{[\sigma,\tau] \times \Omega} \abs{\pbf(\vec{S}_j)} \dd \tv{S_j} \\
  &= \liminf_{j \to \infty} \, \Var(S_j;[\sigma,\tau])
\end{align*}  
since the integrand $\abs{\pbf(\frarg)}$ is positively $1$-homogeneous, convex (as the composition of a convex and a linear map on $\Wedge_{1+k} \R^{1+d}$), and continuous.

Finally, assume that $(S_j) \subset \Irm^\Lip_{1+k}([\sigma,\tau] \times \cl{\Omega})$ and the Lipschitz constants $L_j$ of the scalar maps $t \mapsto \Var(S_j;[\sigma,t]) + \Var(\partial S_j;(\sigma,t))$ are uniformly bounded by $L^* > 0$. Then, by Lemma~\ref{lem:mass}, for almost every $s,t \in (\sigma,\tau)$ and every $j \in \N$,
\begin{align*}
  \tv{S_j}([s,t] \times \R^d)
  &= \Mbf(S_j \restrict ([s,t] \times \R^d)) \\
  &\leq \abs{s-t} \cdot \esssup_{r \in [s,t]} \, \Mbf(S_j(r)) + \Var(S_j;[s,t]) \\
  &\leq (C+L^*) \abs{s-t}.
\end{align*}
Likewise, we obtain $\tv{\partial S_j}([s,t] \times \R^d) \leq (C+L^*) \abs{s-t}$. Then, for the measure $\nu$ defined above it holds that $\nu(\{t\} \times \R^d) = 0$ for \emph{all} $t \in [\sigma,\tau)$. Consequently, the same argument as before yields that $S_j(t) \toweakstar S(t)$ in $\Irm_k(\cl{\Omega})$ at every $t \in [\sigma,\tau)$.

Finally, for $s,t \in [\sigma,\tau]$ it holds that
\[
  \Var(S;[s,t]) \leq \liminf_{j \to \infty} \, \Var(S_j;[s,t]) \leq L \abs{s-t}
\]
by the same argument based on Reshetnyak's theorem as above, where $L := \liminf_{j\to\infty} L_j$; similarly for the boundary variation. In particular, $S \in \Irm^\Lip_{1+k}([\sigma,\tau] \times \cl{\Omega})$ with Lipschitz constant bounded by $L$. 
\end{proof}

\begin{corollary}
Assume that $S_j \toweakstar S$ in $\Irm_{1+k}([\sigma,\tau] \times \cl{\Omega})$ and
\[
  \esssup_{t \in [\sigma,\tau]} \, \bigl( \Mbf(S_j(t)) + \Mbf(\partial S_j(t)) \bigr) \leq C < \infty
\]
for all $j \in \N$. Then, $S_j \toweakstar S$ in BV.
\end{corollary}

\begin{proof}
Since $S_j \toweakstar S$ and $\partial S_j \toweakstar \partial S$ in the sense of measures (which follows from the weak* convergence as currents), we have that $\Mbf(S_j) + \Mbf(\partial S_j) \leq C$ for some constant $C > 0$. Then also $\Var(S_j;[\sigma,\tau]) + \Var(\partial S_j;[\sigma,\tau]) \leq C$ by~\eqref{eq:VarM_est} and so the assumptions of the preceding theorem are satisfied. Since the limit is already determined, we get $S_j \toweakstar S$ in BV.
\end{proof}

\section{Deformations} \label{sc:homdef}

In this section we consider progressive-in-time deformations of (boundaryless) integral currents. To see how one could generalize currents deformed via $\Crm^1$-homotopies (or Lipschitz-homotopies), we first examine the classical situation: Let $\Omega \subset \R^d$ be a bounded Lipschitz domain and let $H \in \Crm^1([0,1] \times \cl{\Omega};\cl{\Omega})$ be a $\Crm^1$-homotopy between the identity and $g \in \Crm^1(\cl{\Omega};\cl{\Omega})$, i.e., $H(0,x) = x$ and $H(1,x) = g(x)$; we also set $\bar{H}(t,x) := (t,H(t,x))$. For $T \in \Irm_k(\cl{\Omega})$ with $\partial T = 0$ define the \term{deformation trajectory}
\[
  S := \bar{H}_*(\dbr{(0,1)} \times T) \in \Irm_{1+k}([0,1] \times \cl{\Omega}),
\]
where we have denoted by $\dbr{(0,1)}$ the canonical current associated with the interval $(0,1)$. Then, by~\eqref{eq:product_bdry} we have
\[
  \partial S = \delta_1 \times g_* T - \delta_0 \times T.
\]
Moreover, since $H$ was assumed to possess $\Crm^1$-regularity,
\[
  S|_t = \delta_t \times H(t,\frarg)_* T,  \qquad t \in [0,1],
\]
and $t \mapsto S|_t$ can be understood as a continuous deformation of $T$ into $g_* T$. 

Unfortunately, the class of $\Crm^1$-homotopies is not closed in a topology suitable for our needs. Furthermore, $\Crm^1$-homotopies do not allow to move overlapping or intersecting parts of currents into different directions since they represent deformations of the underlying space and not of the currents themselves. Our generalization of a deformation is thus based on the deformation trajectory $S$ itself.

\subsection{Homotopies} \label{sc:hom}

We first consider the case of Lipschitz homotopies in more detail. The following lemma also provides a further motivation for our definitions~\eqref{eq:VarM},~\eqref{eq:VarM_bdry} of the variation. Indeed,~\eqref{eq:VarM_hom},~\eqref{eq:VarM_bdry_hom} below show that in the case of essentially injective homotopies (which do not reverse direction and have no overlaps) the variation measures precisely the mass of the pushforward of the space-time current under the spatial projection (the \enquote{slip surface} in the situation of dislocations).

\begin{lemma} \label{lem:injhom}
Let $\Omega \subset \R^d$ and $\Omega' \subset \R^m$ be bounded Lipschitz domains, let $T \in \Irm_k(\cl{\Omega})$, and let $H \in \Lip([0,1] \times \cl{\Omega};\cl{\Omega'})$ be a homotopy that is essentially injective in the sense that
\begin{equation} \label{eq:Hbar_inj}
  \left\{\begin{aligned}  
    &\text{there is a Borel set $N \subset [0,1] \times \supp T$ with $\Hcal^{1+k}(N) = 0$ such that}\\
    &\text{$H$ is injective on $D := ([0,1] \times \supp T) \setminus (N \cup \{ \pbf(D\bar{H}[\ee_0 \wedge \vec{T}]) = 0 \})$,}
  \end{aligned}\right.
\end{equation}
where $\bar{H}(t,x) := (t,H(t,x))$. Set $S_H := \bar{H}_* (\dbr{(0,1)} \times T)$. Then, for all intervals $[\sigma,\tau] \subset [0,1]$,
\begin{align}
  \Var(S_H;[\sigma,\tau]) &= \Mbf(H_*(\dbr{(\sigma,\tau)} \times T)), \label{eq:VarM_hom} \\
  \Var(\partial S_H;(\sigma,\tau)) &= \Mbf(H_*(\dbr{(\sigma,\tau)} \times \partial T)).  \label{eq:VarM_bdry_hom}
\end{align}  
\end{lemma}

\begin{proof}
Let $T = m \, \vec{T} \, \Hcal^k \restrict R$ with a countably $k$-rectifiable carrier set $R$ such that $\Hcal^k(R) < \infty$, and a $\Hcal^k$-measurable and integrable multiplicity function $m \colon R \to \N$. Fix an interval $[\sigma,\tau] \subset [0,1]$ and set $Z := (\sigma,\tau) \times R$, which is a countably $(1+k)$-rectifiable set. Assume furthermore that
\[
  S_H = m_H \, \vec{S}_H \, \Hcal^{1+k} \restrict R_H \in \Irm_{1+k}([0,1] \times \cl{\Omega'}). 
\]
We have for $\omega \in \Dcal^{1+k}(\R^{1+m})$ that
\begin{align*}
  &\dprb{\bar{H}_*(\dbr{(\sigma,\tau)} \times T), \omega} \\
  &\qquad = \int_\sigma^\tau \int_R \dprb{D^Z \bar{H}(t,x)[\ee_0 \wedge \vec{T}(x)], \omega(\bar{H}(t,x))} \; m(x) \dd \Hcal^k(x) \dd t \\
  &\qquad = \int_{R_H \cap ([\sigma,\tau] \times \R^m)} \sum_{\setsub{x}{y = H(t,x)}} m(x) \dprBB{ \frac{D^Z \bar{H}(t,x)[\ee_0 \wedge \vec{T}(x)]}{\abs{D^Z \bar{H}(t,x)[\ee_0 \wedge \vec{T}(x)]}}, \omega(t,y)} \dd \Hcal^{1+k}(t,y)
\end{align*}  
by~\eqref{eq:pushforward} and the area formula, Proposition~\ref{prop:area}. Note that $\abs{D^Z \bar{H}(t,x)[\ee_0 \wedge \vec{T}(x)]}$ is precisely the modulus of the $k$-dimensional Jacobian of $\bar{H}$ at $(t,x)$ with respect to $Z$. It follows, see, e.g.,~\cite[eq.~(7.29)]{KrantzParks08book}), that
\begin{equation} \label{eq:mpSH_sum}
  m_H(t,y) \, \vec{S}_H(t,y) = \sum_{\setsub{x}{y = H(t,x)}} m(x) \frac{D^Z \bar{H}(t,x)[\ee_0 \wedge \vec{T}(x)]}{\abs{D^Z \bar{H}(t,x)[\ee_0 \wedge \vec{T}(x)]}}.
\end{equation}
Define
\[
  \eta(t,x) := \pbf(D^Z \bar{H}(t,x)[\ee_0 \wedge \vec{T}(x)]).
\]
Our assumption~\eqref{eq:Hbar_inj} now implies that whenever $\eta(t,x) \neq 0$, then in~\eqref{eq:mpSH_sum} we have 
\[
  \abs{\pbf(\vec{S}_H(t,y))} \, m_H(t,y) = \sum_{\setsub{x}{y = H(t,x)}} m(x) \frac{\abs{\pbf(D^Z \bar{H}(t,x)[\ee_0 \wedge \vec{T}(x)])}}{\abs{D^Z \bar{H}(t,x)[\ee_0 \wedge \vec{T}(x)]}}
\]
on a set of full measure (there is only one term in the sum). Thus,
\begin{align}
  \Var(S_H;[\sigma,\tau])
  &= \int_{R_H \cap ([\sigma,\tau] \times \R^m)} \abs{\pbf(\vec{S}_H)} \dd \tv{S_H}   \notag\\
  &= \int_{R_H \cap ([\sigma,\tau] \times \R^m)} \sum_{\setsub{x}{y = H(t,x)}} m(x) \frac{\abs{\pbf(D^Z \bar{H}(t,x)[\ee_0 \wedge \vec{T}(x)])}}{\abs{D^Z \bar{H}(t,x)[\ee_0 \wedge \vec{T}(x)]}} \dd \Hcal^{1+k}(t,y)  \notag\\
  &= \int_\sigma^\tau \int_R \abs{\eta(t,x)} \; m(x) \dd \Hcal^k(x) \dd t.\label{eq:VarSH_1}
\end{align}
By similar arguments as before, for $\omega \in \Dcal^{1+k}(\R^m)$ it also holds that
\begin{align*}
  &\dprb{H_*(\dbr{(\sigma,\tau)} \times T), \omega} \\
  &\qquad = \int_\sigma^\tau \int_R \dprb{\eta(t,x), \omega(H(t,x))} \; m(x) \dd \Hcal^k(x) \dd t \\
  &\qquad = \iint_D \dprb{\eta(t,x), \omega(H(t,x))} \; m(x) \dd \Hcal^k(x) \dd t.
\end{align*}
We now find a measurable $k$-covector field $\tilde{\omega} \colon D \to \Wedge^k \R^m$ with $\abs{\tilde{\omega}} \leq 1$ and $\dpr{\eta(t,x),\tilde{\omega}(t,x)} = \abs{\eta(t,x)}$. Then, by~\eqref{eq:Hbar_inj}, there exists a measurable $k$-covector field $\hat{\omega} \colon \cl{\Omega} \to \Wedge^k \R^m$ satisfying $\abs{\hat{\omega}} \leq 1$ and
\[
  \dprb{\eta(t,x),\hat{\omega}(H(t,x))} = \dprb{\eta(t,x),\tilde{\omega}(t,x)} = \abs{\eta(t,x)}
\]
for $(t,x) \in D$. By a standard smoothing argument we thus obtain
\[
  \Mbf(H_*(\dbr{(\sigma,\tau)} \times T)) \geq
  \int_\sigma^\tau \int_R \abs{\eta(t,x)} \; m(x) \dd \Hcal^k(x) \dd t
\]
and the other inequality \enquote{$\leq$} is easily seen to be true as well. Consequently, using this equality in~\eqref{eq:VarSH_1},
\[
  \Var(S_H;[\sigma,\tau]) = \Mbf(H_*(\dbr{(\sigma,\tau)} \times T)).
\]
This shows~\eqref{eq:VarM_hom}; the boundary estimate~\eqref{eq:VarM_bdry_hom} follows in the same way.
\end{proof}

\begin{remark}
The proof shows that one could also require the slightly weaker condition
\[
  \left\{\begin{aligned}  
    &\text{there is a Borel set $N \subset [0,1] \times \supp T$ with $\Hcal^{1+k}(N) = 0$ such that}\\
    &\text{$H$ is injective on $D := ([0,1] \times \supp T) \setminus (N \cup \{ \pbf(D^Z\bar{H}[\ee_0 \wedge \vec{T}]) = 0 \})$,}
  \end{aligned}\right.
\]
where $Z := (0,1) \times R$ (with $R$ being the carrier set of $T$), instead of~\eqref{eq:Hbar_inj}.
\end{remark}

The prototypical class of deformation trajectories is defined via affine homotopies:

\begin{lemma} \label{lem:aff_hom}
Let $\Omega \subset \R^d$ and $\Omega' \subset \R^m$ be bounded Lipschitz domains and let $H$ be an affine homotopy between $f,g \in \Lip(\cl{\Omega};\cl{\Omega'})$, i.e.,
\[
  H(t,x) := (1-t) f(x) + t g(x),  \qquad (t,x) \in [0,1] \times \cl{\Omega}.
\]
Let $T \in \Irm_k(\cl{\Omega})$ and set $S_H := \bar{H}_* (\dbr{(0,1)} \times T)$, where $\bar{H}(t,x) := (t,H(t,x))$. Then,
\begin{equation} \label{eq:aff_hom_SH}
  S_H \in \Irm^\Lip_{1+k}([0,1] \times \cl{\Omega'})
\end{equation}
and, for all intervals $[\sigma,\tau] \subset [0,1]$ and almost every $t \in [0,1]$,
\begin{align}
  \Var(S_H;[\sigma,\tau]) &\leq \norm{g-f}_\infty \cdot V^k(f,g,T) \cdot \abs{\sigma-\tau},   \label{eq:aff_hom_Var} \\
  \Var(\partial S_H;(\sigma,\tau)) &\leq \norm{g-f}_\infty \cdot V^{k-1}(f,g,\partial T) \cdot \abs{\sigma-\tau},  \label{eq:aff_hom_bVar} \\
  \Mbf(S_H(t)) &\leq V^k(f,g,T), \label{eq:aff_hom_M} \\
  \Mbf(\partial S_H(t)) &\leq V^{k-1}(f,g,\partial T), \label{eq:aff_hom_bM}
\end{align}
where, for $\ell = k-1,k$,
\[
  V^\ell(f,g,T) := \int \abs{Df}^\ell + \abs{Dg}^\ell \dd \tv{T}
  \leq \bigl(\norm{Df}_{\Lrm^\infty}^\ell + \norm{Dg}_{\Lrm^\infty}^\ell \bigr) \Mbf(T).
\]
and the $\Lrm^\infty$-norms may be taken over the support of $T$.
\end{lemma}

\begin{proof}
We use the same notation as in the proof of Lemma~\ref{lem:injhom}. Recalling~\eqref{eq:mpSH_sum} (which holds independently of the injectivity hypothesis~\eqref{eq:Hbar_inj}), we can estimate
\begin{equation} \label{eq:mpSH_est}
  \abs{\pbf(\vec{S}_H(t,y))} \, m_H(t,y) \leq \sum_{\setsub{x}{y = (1-t)f(x) + tg(x)}} m(x) \frac{\abs{\pbf(D^Z \bar{H}(t,x)[\ee_0 \wedge \vec{T}(x)])}}{\abs{D^Z \bar{H}(t,x)[\ee_0 \wedge \vec{T}(x)]}} .
\end{equation}
A computation shows
\[
  D^Z \bar{H}(t,x)[\ee_0 \wedge \vec{T}(x)]
  = \begin{pmatrix} 1 \\ g(x)-f(x) \end{pmatrix} \wedge \biggl((1-t) \begin{pmatrix} 0 \\ Df(x) \end{pmatrix} + t \begin{pmatrix} 0 \\ Dg(x) \end{pmatrix} \biggr)[\vec{T}(x)]
\]
and then
\begin{align*}
  \absb{\pbf(D^Z \bar{H}(t,x)[\ee_0 \wedge \vec{T}(x)])}
  &= \absb{(g(x)-f(x)) \wedge ((1-t) Df(x) + t Dg(x)) [\vec{T}(x)]} \\
  &\leq \norm{g-f}_\infty \cdot (\abs{Df(x)}^k + \abs{Dg(x)}^k).
\end{align*}
So,
\begin{align*}
  &\Var(S_H;[\sigma,\tau]) \\
  &\qquad = \int_{R_H \cap ([\sigma,\tau] \times \R^m)} \abs{\pbf(\vec{S}_H(t,y))} \, m_H(t,y) \dd \Hcal^{1+k}(t,y) \\
  &\qquad \leq \int_{R_H \cap ([\sigma,\tau] \times \R^m)} \sum_{\setsub{x}{y = (1-t)f(x) + tg(x)}} m(x) \frac{\abs{\pbf(D^Z \bar{H}(t,x)[\ee_0 \wedge \vec{T}(x)])}}{\abs{D^Z \bar{H}(t,x)[\ee_0 \wedge \vec{T}(x)]}} \dd \Hcal^{1+k}(t,y) \\
  &\qquad \leq \int_{R_H \cap ([\sigma,\tau] \times \R^m)} \sum_{\setsub{x}{y = (1-t)f(x) + tg(x)}} m(x) \frac{\norm{g-f}_\infty \cdot (\abs{Df(x)}^k + \abs{Dg(x)}^k)}{\abs{D^Z \bar{H}(t,x)[\ee_0 \wedge \vec{T}(x)]}} \dd \Hcal^{1+k}(t,y) \\
  &\qquad = \int_\sigma^\tau \int_R \norm{g-f}_\infty \cdot (\abs{Df(x)}^k + \abs{Dg(x)}^k) \; m(x) \dd \Hcal^k(x) \dd t \\
  &\qquad = \norm{g-f}_\infty \cdot \biggl( \int \abs{Df}^k + \abs{Dg}^k \dd \tv{T} \biggr) \cdot \abs{\sigma-\tau},
\end{align*}
where we have used the area formula again in the second-to-last equality. This shows~\eqref{eq:aff_hom_Var}.

For the boundary variation $\Var(\partial S_H;(\sigma,\tau))$ we observe via~\eqref{eq:product_bdry},~\eqref{eq:pushforward_bdry} that
\[
  \partial S_H = \delta_1 \times g_* T - \delta_0 \times f_* T - \bar{H}_* (\dbr{(0,1)} \times \partial T).
\]
We can argue in a similar fashion to above to obtain~\eqref{eq:aff_hom_bVar} (note that the interval is open, so that the endpoint terms are not counted). For~\eqref{eq:aff_hom_M} we use that for almost every $t \in (0,1)$ it holds that
\[
  S_H(t) = H(t,\frarg)_* T.
\]
This follows from the cylinder formula~\eqref{eq:cylinder}. Then, for $\omega \in \Dcal^k(\R^m)$,
\[
  \dprb{S_H(t),\omega} = \int_R \dprb{D^R H(t,\frarg)[\vec{T}(x)], \omega(H(t,x))} \; m(x) \dd \Hcal^k(x).
\]
Taking the supremum over all $\omega \in \Dcal^k(\R^m)$ with $\abs{\omega} \leq 1$ and employing a similar estimate as above yields~\eqref{eq:aff_hom_M}; likewise for~\eqref{eq:aff_hom_bM}. Then, also~\eqref{eq:aff_hom_SH} follows.
\end{proof}

\begin{remark}
Note that for an affine homotopy from $f$ to $g$ as in the preceding lemma,
\[
  \eta(t,x) = \pbf(D^Z \bar{H}(t,x)[\ee_0 \wedge \vec{T}(x)]) = (g(x)-f(x)) \wedge ((1-t) Df(x) + t Dg(x)) [\vec{T}(x)],
\]
which is zero in particular where $f = g$ (that is, where the affine homotopy \enquote{stands still}). So, in this case, the assumption~\eqref{eq:Hbar_inj} in Lemma~\ref{lem:injhom} is implied by the more restrictive, but easier to check, condition
\[
  \left\{\begin{aligned}  
    &\text{there is a Borel set $N \subset [0,1] \times \supp T$ with $\Hcal^{1+k}(N) = 0$ such that}\\
    &\text{$H$ is injective on $D := ([0,1] \times \supp T) \setminus (N \cup \{ f = g \})$.}
  \end{aligned}\right.
\]
\end{remark}

\subsection{Operations on space-time currents}

Before we come to the main results of this section, it is convenient to define the concatenation and reversal of space-time currents with boundaryless traces at the start end end points:

\begin{lemma} \label{lem:concat}
Let $S_1,S_2 \in \Irm_{1+k}([0,1] \times \cl{\Omega})$ with
\[
  \partial S_1 = \delta_1 \times T_1 - \delta_0 \times T_0, \qquad
  \partial S_2 = \delta_1 \times T_2 - \delta_0 \times T_1,
\]
where $T_0,T_1,T_2 \in \Irm_k(\cl{\Omega})$ with $\partial T_0 = \partial T_1 = \partial T_2 = 0$. Then, there is $S_2 \circ S_1 \in \Irm_{1+k}([0,1] \times \cl{\Omega})$, called the \term{concatenation} of $S_1,S_2$, with
\[
  \partial (S_2 \circ S_1) = \delta_1 \times T_2 - \delta_0 \times T_0
\]
and
\begin{align*}
  \Var(S_2 \circ S_1) &= \Var(S_1) + \Var(S_2),  \\
  \esssup_{t \in [0,1]} \, \Mbf((S_2 \circ S_1)(t)) &= \max \biggl\{ \esssup_{t \in [0,1]} \, \Mbf(S_1(t)), \; \esssup_{t \in [0,1]} \, \Mbf(S_2(t)) \biggr\}.
\end{align*} 
Furthermore, if $S_1,S_2 \in \Irm^\Lip_{1+k}([0,1] \times \cl{\Omega})$, then also $S_2 \circ S_1 \in \Irm^\Lip_{1+k}([0,1] \times \cl{\Omega})$.
\end{lemma}

\begin{proof}
We set
\[
  S := a^1_* S_1 + a^2_* S_2,  \qquad
  a^i(t,x) := \biggl( \frac{i-1}{2} + \frac{t}{2}, x \biggr).
\]
Then, all claimed properties follow directly from Lemma~\ref{lem:rescale}.
\end{proof}

\begin{lemma} \label{lem:reversal}
Let $S \in \Irm_{1+k}([0,1] \times \cl{\Omega})$ with
\[
  \partial S = \delta_1 \times T_1 - \delta_0 \times T_0,
\]
where $T_0,T_1 \in \Irm_k(\cl{\Omega})$ with $\partial T_0 = \partial T_1 = 0$. Then, there is $S^{-1} \in \Irm_{1+k}([0,1] \times \cl{\Omega})$, called the \term{reversal} of $S$, with
\[
  \partial S^{-1} = \delta_1 \times T_0 - \delta_0 \times T_1
\]
and
\[
  \Var(S^{-1}) = \Var(S),  \qquad
  \esssup_{t \in [0,1]} \, \Mbf(S^{-1}(t)) = \esssup_{t \in [0,1]} \, \Mbf(S(t)).
\]
Furthermore, if $S \in \Irm^\Lip_{1+k}([0,1] \times \cl{\Omega})$, then also $S^{-1} \in \Irm^\Lip_{1+k}([0,1] \times \cl{\Omega})$.
\end{lemma}

\begin{proof}
We set
\[
  S^{-1} := a_* S,  \qquad
  a(t,x) := (1-t,x)
\]
and again conclude by Lemma~\ref{lem:rescale}.
\end{proof}

\subsection{Deformation theorem} \label{sc:deformation_thm}

In this section we establish a version of the deformation theorem (see~\cite[4.2.9]{Federer69book} or~\cite[Section~7.7]{KrantzParks08book}) that is adapted to our BV-theory of integral currents. Let us emphasize that this theorem requires the current being approximated to be integral and boundaryless. Also recall our standing assumption that $\Omega \subset \R^d$ is a bounded Lipschitz domain.

\begin{theorem}[Deformation theorem] \label{thm:defo}
Let $T \in \Irm_k(\cl{\Omega})$ with $\partial T = 0$. Then, for all $\rho > 0$ there exists $S \in \Irm^\Lip_{1+k}([0,1] \times \cl{\Omega'})$, where $\Omega' := \Omega + B(0,(\sqrt{d}+1)\rho)$, such that
\[
  \partial S = \delta_1 \times P - \delta_0 \times T, \qquad
  P = \sum_{F \in \Fcal_k(\rho)} p_F \dbr{F},  \qquad
  \partial P = 0.
\]
Here, $\dbr{F}$ is the integral current associated to an oriented $k$-face $F \in \Fcal_k(\rho)$ of one of the cubes $\rho z + (0,\rho)^d$ (with unit multiplicity and a fixed choice of orientation), $z \in \Z^d$, and $p_F \in \Z$. Moreover,
\begin{align*}
  \Mbf(P) &\leq C \Mbf(T), \\
  \Var(S) &\leq C \rho \Mbf(T), \\
  \esssup_{t \in [0,1]} \, \Mbf(S(t)) &\leq C \Mbf(T).
\end{align*}
Here, the constant $C > 0$ depends only on the dimensions.
\end{theorem}

\begin{proof}
It suffices to prove the theorem for $\rho = 1$; the general case is reduced to $\rho = 1$ by scaling. Indeed, setting $r^\alpha(x) := \alpha x$ for $\alpha > 0$ and $x \in \R^d$, we may apply the result in the version for $\rho = 1$ to $\tilde{T} := r^{1/\rho}_* T$ (in a suitable domain) to obtain $\tilde{P},\tilde{S}$ as in the statement of the theorem for $\rho = 1$. Then, set $P := r^{\rho}_* \tilde{P}$, $S := r^{\rho}_* \tilde{S}$ (or, more verbosely, $S := [(t,x) \mapsto (t,\rho x)]_* \tilde{S}$). These $P,S$ satisfy the conclusion of the theorem for our $\rho$ since $P,T$ and $S(t)$ (for a.e.\ $t \in [0,1]$) have the same dimension $k$ and $S$ has dimension $1+k$, whereby
\[
  \Var(S) = \rho^{1+k} \Var(\tilde{S}) \leq C \rho^{1+k} \Mbf(\tilde{T}) = C \rho \Mbf(T).
\]
One quick way to see the first equality is to observe that
\[
  \Var(S)
  = \rho^{1+k} \Var\bigl( [(t,x) \mapsto (t/\rho,x/\rho)]_* S;[0,1/\rho] \bigr)
  = \rho^{1+k} \Var(\tilde{S})
\]
by the area formula and Lemma~\ref{lem:rescale}.

So, in the following let $\rho = 1$. Inspecting the proof of the standard deformation theorem, in the version of~\cite[Sections~7.7,~7.8]{KrantzParks08book}, say, we observe that in the present situation of \emph{boundaryless} integral currents the proof proceeds by constructing a homotopy from $T$ to a $P$ of the form
\[
  P = \sum_{F \in \Fcal_k(1)} p_F \dbr{F}
\]
with
\[
  \partial P = 0,  \qquad
  \Mbf(P) \leq C \Mbf(T).
\]
We remark in particular that in the last step of the proof of the deformation theorem (as in~\cite[Section~7.8]{KrantzParks08book}) we do not need to modify the retraction onto any $k$-face since $\partial T = 0$ (by the constancy theorem, see~\cite[Proposition~7.3.5]{KrantzParks08book}) and $P$ is indeed a homotopical image of $T$. The homotopy constructed is seen to be the concatenation of two affine homotopies: The first affine homotopy, call it $H_1$, goes from the identity to a translation $t^a(x) := x + a$ ($\abs{a} < 1$). The second affine homotopy, $H_2$, goes from the identity to the \enquote{radial} retraction $\psi$ onto the $k$-skeleton (defined in~\cite[Section~7.7]{KrantzParks08book}).

We have $H_1(1,\frarg)_* T = t^a_* T$ and
\[
  S_{H_1} := (\bar{H}_1)_* T \in \Irm^\Lip_{1+k}([0,1] \times \cl{\Omega + B(0,1)}),
\]
where $\bar{H}_1(t,x) := (t,H_1(t,x))$. From Lemma~\ref{lem:aff_hom} we obtain
\[
  \Var(S_{H_1}) \leq C \Mbf(T).
\]
Moreover, it can be shown (see~\cite[top of p.~218]{KrantzParks08book}) that $a$ may be chosen such that
\[
  \int \abs{D\psi}^k \dd \tv{t^a_* T} \leq C \Mbf(T).
\]
Thus, from Lemma~\ref{lem:aff_hom} we get for
\[
  S_{H_2} := (\bar{H}_2)_* [t^a_* T] \in \Irm^\Lip_{1+k}([0,1] \times \cl{\Omega + B(0,1 + \sqrt{d})}),
\]
where $\bar{H}_2(t,x) := (t,H_2(t,x))$, that also
\[
  \Var(S_{H_2}) \leq C \Mbf(T).
\]

Once we concatenate $S_{H_1}$ and $S_{H_2}$ via Lemma~\ref{lem:concat}, we obtain that for
\[
  S := S_{H_2} \circ S_{H_1} \in \Irm^\Lip_{1+k}([0,1] \times \cl{\Omega'})
\]
it holds that $\partial S = \delta_1 \times P - \delta_0 \times T$ and
\[
  \Var(S) \leq C \Mbf(T).
\]
The statement about the essential mass bound on $S(t)$ also follows from the estimates of Lemmas~\ref{lem:aff_hom},~\ref{lem:concat}. This finishes the proof.
\end{proof}

As a corollary, we obtain the following version of the isoperimetric inequality:

\begin{theorem}[Isoperimetric inequality] \label{thm:isop}
Let $T \in \Irm_k(\cl{\Omega})$, $k \geq 1$, with $\partial T = 0$. Then, there exists $S \in \Irm^\Lip_{1+k}([0,1] \times \cl{\Omega'})$, where $\Omega' := \Omega + B(0,C\Mbf(T)^{1/k})$, such that
\[
  \partial S = - \delta_0 \times T
\]
and
\[
  \Var(S;[0,1]) \leq C \Mbf(T)^{(k+1)/k},  \qquad
  \esssup_{t \in [0,1]} \, \Mbf(S(t)) \leq C \Mbf(T).
\]
Here, the constant $C > 0$ depends only on the dimensions.
\end{theorem}

\begin{proof}
The proof is similar to the one for the classical isoperimetric inequality and follows immediately from the deformation theorem: Assuming that $T \neq 0$, we let $P,S$ as in the deformation theorem with
\[
  \rho := [2C \Mbf(T)]^{1/k},
\]
where $C > 0$ is the constant from said theorem. By a scaling argument, $\Mbf(P) = N(\rho)\rho^k$ for some nonnegative integer $N(\rho)$. From the estimates in the deformation theorem we have $\Mbf(P) \leq C \Mbf(T)$ and thus
\[
  N(\rho) \cdot 2C \Mbf(T) = \Mbf(P) \leq C \Mbf(T).
\]
So, $2N(\rho) \leq 1$, whereby $N(\rho) = 0$, and hence $P = 0$. This immediately yields all the claimed statements.
\end{proof}

\section{Deformation distance} \label{sc:defdist}

We now define a metric measuring the distance between two boundaryless integral $k$-currents via progressive-in-time deformations, namely Lip-integral currents. In all of the following, $\Omega \subset \R^d$ is a bounded Lipschitz domain.

For $T_0, T_1 \in \Irm_k(\cl{\Omega})$ with $\partial T_0 = \partial T_1 = 0$, the \term{(Lipschitz) deformation distance} between $T_0$ and $T_1$ is
\[
  \dist_{\Lip,\cl{\Omega}}(T_0,T_1) := \inf \setB{ \Var(S) }{ \text{$S \in \Irm^\Lip_{1+k}([0,1] \times \cl{\Omega})$ with $\partial S = \delta_1 \times T_1 - \delta_0 \times T_0$} }.
\]
That $\dist_{\Lip,\cl{\Omega}}(\frarg,\frarg) \colon \Irm_k(\cl{\Omega}) \times \Irm_k(\cl{\Omega}) \to [0,\infty]$ is positive definite, symmetric, and obeys the triangle inequality follows immediately from Lemmas~\ref{lem:concat},~\ref{lem:reversal} and the fact that $\Var(S) = 0$ for $S \in \Irm^\Lip_{1+k}([0,1] \times \cl{\Omega})$ with $\partial S = \delta_1 \times T_1 - \delta_0 \times T_0$ implies that $T_0 = T_1$. We remark that $\dist_{\Lip,\cl{\Omega}}(\frarg,\frarg)$ is not necessarily finite if $\cl{\Omega}$ has holes that can be detected by boundaryless integral $k$-currents.

\subsection{Equivalence theorem}

With regard to the notion of convergence induced by the (Lipschitz) deformation distance, we have the following result:

\begin{theorem}[Equivalence theorem] \label{thm:equiv}
For every $M > 0$ and $T_j,T$ ($j \in \N$) in the set
\[
  \setb{ T \in \Irm_k(\cl{\Omega}) }{ \partial T = 0, \; \Mbf(T) \leq M }
\]
the following equivalence holds (as $j \to \infty$):
\[
  \dist_{\Lip,\cl{\Omega}}(T_j,T) \to 0  \qquad\text{if and only if}\qquad   T_j \toweakstar T \quad \text{in $\Irm_k(\cl{\Omega})$}.
\]
Moreover, in this case, for all $j$ from a subsequence of the $j$'s, there are $S_j \in \Irm^\Lip_{1+k}([0,1] \times \cl{\Omega})$ with
\[
  \partial S_j = \delta_1 \times T - \delta_0 \times T_j, \qquad
  \dist_{\Lip,\cl{\Omega}}(T_j,T) \leq \Var(S_j) \to 0,
\]
and
\begin{equation} \label{eq:equiv_mass}
  \limsup_{j\to\infty} \, \; \esssup_{t \in [0,1]} \, \Mbf(S_j(t)) \leq C \cdot \limsup_{\ell \to \infty} \, \Mbf(T_\ell).
\end{equation}
Here, the constant $C > 0$ depends only on the dimensions and on $\Omega$.
\end{theorem}

\begin{proof}
For the first direction, assume $\dist_{\Lip,\cl{\Omega}}(T_j,T) \to 0$. By~\eqref{eq:PoincareF} we have $\Fbf(T_j-T) \leq \Var(S)$ for any $S \in \Irm^\Lip_{1+k}([0,1] \times \cl{\Omega})$ with $\partial S = \delta_1 \times T_j - \delta_0 \times T_0$. Here, we remark that we do not require the boundary variation since the only contributions to $\Var(\partial S)$ are at the endpoints $0,1$, but we can restrict to the open interval $(0,1)$ and use the right and left limits, cf.~\eqref{eq:partialS_trace}. Thus,
\[
  \Fbf(T_j-T) \leq \dist_{\Lip,\cl{\Omega}}(T_j,T) \to 0.
\]
Then, the claim $T_j \toweakstar T$ follows from Proposition~\ref{prop:flat}, or directly as follows: For $\omega \in \Dcal^k(\R^d)$,
\[
  \absb{\dprb{T_j - T,\omega}}
  \leq \Fbf(T_j-T) \cdot \max \bigl\{ \norm{\omega}_\infty, \norm{d\omega}_\infty \bigr\}
  \leq \dist_{\Lip,\cl{\Omega}}(T_j,T) \cdot \max \bigl\{ \norm{\omega}_\infty, \norm{d\omega}_\infty \bigr\}
  \to 0.
\]

For the other direction, assume $T_j \toweakstar T$ in $\Irm_k(\cl{\Omega})$ with $\partial T_j = \partial T = 0$ and $M := \sup_j \Mbf(T_j) < \infty$. We need to show that
\begin{equation} \label{eq:distH0}
  \dist_{\Lip,\cl{\Omega}}(T_j,T) \to 0.
\end{equation}

The first step is to observe that for all $N \in \N$ sufficiently large there exists a finite collection $\Pbf_N \subset \Irm_k(\cl{\Omega})$ such that for all $\hat{T} \in \Irm_k(\cl{\Omega})$ with $\partial \hat{T} = 0$ and $\Mbf(\hat{T}) \leq M$ it holds that
\begin{equation} \label{eq:distH_small}
   \text{$\dist_{\Lip,\cl{\Omega}}(\hat{T},P) < C2^{-N}$ for some $P \in \Pbf_N$,}
\end{equation}
where the constant $C > 0$ and the lower bound for $N$ depend only on the dimensions and the domain $\Omega$. We first claim that~\eqref{eq:distH_small} holds (with $C = 1$ and for all $N \in \N$) for $\dist_{\Lip,\cl{\Omega'}}$, where we have set $\Omega' := \Omega + B(0,(\sqrt{d}+1)\rho) \Supset \Omega$ is as in our deformation theorem, Theorem~\ref{thm:defo}, with $\rho := 2^{-N}/(C_{d,k}M)$ (with $C_{d,k}$ the constant from the deformation theorem). Indeed, for $\Pbf_N$ we take the collection of all polyhedral chains $P$ that can possibly satisfy the conclusion of the deformation theorem for a $\hat{T}$ as above, which is clearly a finite set. Thus,~\eqref{eq:distH_small} is established in $\cl{\Omega'}$. 

Next, for $N$ sufficiently large (how large only depending on $\Omega$), we may retract $\cl{\Omega'}$ to $\cl{\Omega}$. In this context recall that $\Omega$ is always assumed to be a bounded Lipschitz domain and hence a Lipschitz neighborhood retract, see Remark~\ref{rem:retract}. Thus,~\eqref{eq:distH_small} also holds for $\dist_{\Lip,\cl{\Omega}}$ and with $\Pbf_N$ containing the retracts of the polyhedral chains. Note that the retraction itself only contributes a bounded factor to the estimate of the variation.

Returning to our sequence $(T_j)$, for every $N \in \N$ sufficiently large we find a $P \in \Pbf_N$ such that $\dist_{\Lip,\cl{\Omega}}(T_j,P) < C2^{-N}$ for infinitely many $j$'s. Applying this argument repeatedly and selecting a subsequence at every step (such that the constraint holds for all elements of that subsequence), we may find a diagonal subsequence, still denoted by $(T_j)$, such that $\dist_{\Lip,\cl{\Omega}}(T_\ell,P_j) < 2^{-(j+1)}$ for all $\ell \geq j$ and a $P_j \in \bigcup_N \Pbf_N \subset \Irm_k(\cl{\Omega})$ (by the construction above $P_j$ is the Lipschitz retract of a polyhedral chain). Then, via the triangle inequality,
\[
  \dist_{\Lip,\cl{\Omega}}(T_j,T_{j+1}) < 2^{-j}.
\]
Hence, there exists an $R_j \in \Irm^\Lip_{1+k}([0,1] \times \cl{\Omega})$ with
\[
  \partial R_j = \delta_1 \times T_{j+1} - \delta_0 \times T_j,  \qquad
  \Var(R_j;[0,1]) < 2^{-j}.
\]
Using the space-time currents constructed in the proof of the deformation theorem as witnesses for $\dist_{\Lip,\cl{\Omega}}(T_j,P_j) < 2^{-(j+1)}$ and $\dist_{\Lip,\cl{\Omega}}(P_j,T_{j+1}) < 2^{-(j+1)}$ and concatenating them via Lemma~\ref{lem:concat} to obtain $R_j$, we may further require
\[
  \esssup_{t \in [0,1]} \, \Mbf(R_j(t)) \leq C \cdot \max \bigl\{ \Mbf(T_j),\Mbf(T_{j+1}) \bigr\}.
\]

For the concatenation of the $R_\ell$ for $\ell = j,\ldots,j+m-1$, that is,
\[
  S_j^m := R_{j+m-1} \circ R_{j+m-2} \circ \cdots \circ R_j,
\]
see again Lemma~\ref{lem:concat}, it holds that
\[
  \partial S_j^m = \delta_1 \times T_{j+m} - \delta_0 \times T_j.
\]
and
\begin{align*}
  \Var(S_j^m;[0,1]) &= \sum_{\ell=0}^{m-1} \Var(R_{j+\ell};[0,1]) \leq 2^{-j+1}, \\
  \Var(\partial S_j^m;[0,1]) &= \Mbf(T_j) + \Mbf(T_{j+m}) \leq 2M, \\
  \esssup_{t \in [0,1]} \, \Mbf(S_j^m(t)) &\leq C \cdot \max_{\ell = j,\ldots,j+m} \, \Mbf(T_\ell) \leq C \cdot \sup_{\ell \geq j} \, \Mbf(T_\ell).
\end{align*}
Moreover, via Lemma~\ref{lem:rescale} we may rescale $S_j^m$ in time (which we do not make explicit in our notation) to assume
\[
  \Var(S_j^m;[0,t]) = t \Var(S_j^m;[0,1]),  \qquad t \in [0,1].
\]
In this way, also the Lipschitz constants of $S_j^m$ are uniformly in $m$ bounded by $2^{-j+1}$.

We now pass to the limit $m \to \infty$. Via Theorem~\ref{thm:current_Helly} this yields $S_j \in \Irm^\Lip_{1+k}([0,1] \times \cl{\Omega})$ with
\[
  \partial S_j = \delta_1 \times \Bigl( \wslim_{m\to\infty} T_{j+m} \Bigr) - \delta_0 \times T_j = \delta_1 \times T - \delta_0 \times T_j
\]
and
\begin{align*}
  \Var(S_j;[0,1]) &\leq 2^{-j+1}, \\
  \esssup_{t \in [0,1]} \, \Mbf(S_j(t)) &\leq C \cdot \sup_{\ell \geq j} \, \Mbf(T_\ell).
\end{align*}
Our $S_j$ is admissible in the definition of the metric $\dist_{\Lip,\cl{\Omega}}(\frarg,\frarg)$ and so,
\[
  \dist_{\Lip,\cl{\Omega}}(T_j,T) \leq \Var(S_j) \to 0  \qquad\text{as $j \to \infty$.}
\]
In this way we can find for every subsequence of the original sequence $(T_j)$ (before taking the repeated subsequences above) a further subsequence that converges in the $\dist_{\Lip,\cl{\Omega}}$-metric to $T$. Hence, also $\dist_{\Lip,\cl{\Omega}}(T_j,T) \to 0$ for the original sequence, proving our claim~\eqref{eq:distH0}.  

Finally, taking the upper limit of the mass estimate,
\[
  \limsup_{j\to\infty} \, \; \esssup_{t \in [0,1]} \, \Mbf(S_j(t)) \leq C \cdot \limsup_{\ell \to \infty} \, \Mbf(T_\ell).
\]
This finishes the proof.
\end{proof}

\subsection{Equality theorem}

Finally, we investigate the relationship of the deformation distance to the \term{integral homogeneous Whitney flat norm} in the bounded Lipschitz domain $\Omega \subset \R^d$, which for $T \in \Irm_k(\cl{\Omega})$ with $\partial T = 0$ is defined as
\[
  \Fbb_{\cl{\Omega}}(T) := \inf \, \setB{ \Mbf(Q) }{ \text{$Q \in \Irm_{k+1}(\cl{\Omega})$ with $\partial Q = T$} }.
\]

We first record the following lemma on the relationship between the different notions of convergences we have encountered so far.

\begin{lemma} \label{lem:conv}
For every $M > 0$ and $T_j,T$ ($j \in \N$) in the set
\[
  \setb{ T \in \Irm_k(\cl{\Omega}) }{ \partial T = 0, \; \Mbf(T) \leq M }
\]
the following are equivalent (as $j \to \infty$):
\begin{enumerate}[(i)]
\item $\dist_{\Lip,\cl{\Omega}}(T_j,T) \to 0$; \label{it:conv_dist}
\item $T_j \toweakstar T$; \label{it:conv_w*}
\item $\Fbf(T-T_j) \to 0$; \label{it:conv_F}
\item $\Fbb_{\cl{\Omega}}(T-T_j) \to 0$. \label{it:conv_F0}
\end{enumerate}
\end{lemma}

\begin{proof}
The equivalence of~\ref{it:conv_dist} and~\ref{it:conv_w*} was proved in Theorem~\ref{thm:equiv}, while the equivalence of~\ref{it:conv_w*} and~\ref{it:conv_F} is the content of Proposition~\ref{prop:flat}. In fact, the proof of (the trivial direction of) Theorem~\ref{thm:equiv} even yields that~\ref{it:conv_dist} implies~\ref{it:conv_F0}. Finally, $\Fbf \leq \Fbb_{\cl{\Omega}}$, so~\ref{it:conv_F0} implies~\ref{it:conv_F} and we have closed the circle of implications.
\end{proof}

\begin{remark}
For the \emph{global} Whitney flat norms $\Fbf$ and $\Fbb := \Fbb_{\R^d}$ one may observe the inequalities
\begin{equation} \label{eq:F0F}
  \Fbf(T) \leq \Fbb(T) \leq C (\Fbf(T) + \Fbf(T)^{(k+1)/k})
\end{equation}
for all $T \in \Irm_k(\R^d)$ with $\partial T = 0$, where $C > 0$ is a dimensional constant. Indeed, the first inequality is trivial and for the second one writes $T = \partial Q + R$ for $Q \in \Irm_{k+1}(\R^d)$, $R \in \Irm_k(\R^d)$ with $\Mbf(Q) + \Mbf(R) \leq 2\Fbf(T)$. Then, $\partial R = \partial T - \partial \partial Q = 0$, and so, by the classical isoperimetric inequality (see, e.g.,~\cite[Theorem~7.9.1]{KrantzParks08book} or~\cite[4.2.10]{Federer69book}), there is $Q' \in \Irm_{k+1}(\R^d)$ with $\partial Q' = R$ and $\Mbf(Q') \leq C \Mbf(R)^{(k+1)/k}$. For $\tilde{Q} := Q + Q'$ we then have $\partial \tilde{Q} = T - R + \partial Q' = T$ and thus
\[
  \Fbb(T)
  \leq \Mbf(\tilde{Q}) \leq \Mbf(Q) + C \Mbf(R)^{(k+1)/k}
  \leq C (\Fbf(T) + \Fbf(T)^{(k+1)/k})
\]
with a different (but still dimensional) constant $C > 0$. This shows~\eqref{eq:F0F}. However, the second inequality in~\eqref{eq:F0F} with $\Fbb_{\cl{\Omega}}$ in place of the global $\Fbb$ may not hold.
\end{remark}

We can now prove the main result of this section, namely that the integral homogeneous Whitney flat distance is equal to the deformation distance.

\begin{theorem}[Equality theorem] \label{thm:equal}
For $T_0, T_1 \in \Irm_k(\cl{\Omega})$ with $\partial T_0 = \partial T_1 = 0$ it holds that
\[
  \dist_{\Lip,\cl{\Omega}}(T_0,T_1) = \Fbb_{\cl{\Omega}}(T_1 - T_0).
\]
\end{theorem}

\begin{proof}
First, the inequality
\[
  \dist_{\Lip,\cl{\Omega}}(T_0,T_1) \geq \Fbb_{\cl{\Omega}}(T_1 - T_0)
\]
follows easily by taking any $S \in \Irm^\Lip_{1+k}([0,1] \times \cl{\Omega})$ with $\partial S = \delta_1 \times T_1 - \delta_0 \times T_0$, setting $Q := \pbf_* S \in \Irm_{k+1}(\cl{\Omega})$, and observing that $\partial Q = T_1 - T_0$ as well as $\Var(S) \geq \Mbf(Q)$ as in~\eqref{eq:FVar_est}. Taking the infimum over all such $S$ yields the above inequality.

For the other inequality, let $\eps > 0$. We first observe by our deformation theorem, Theorem~\ref{thm:defo}, that for $i = 0,1$ there exist $U_i \in \Irm^\Lip_{1+k}([0,1] \times \cl{\Omega'})$, where $\Omega' := \Omega + B(0,(\sqrt{d}+1)\eps)$, with
\[
  \partial U_i = \delta_1 \times P_i - \delta_0 \times T_i, \qquad
  P_i = \sum_{F \in \Fcal_k(\eps)} p^{(i)}_F \dbr{F}  \qquad
  \partial P_i = 0,
\]
such that
\begin{align*}
  \Mbf(P_i) &\leq C \Mbf(T_i), \\
  \Fbb_{\cl{\Omega'}}(T_i - P_i) \leq \Var(U_i) &\leq C \eps \Mbf(T_i), \\
  \esssup_{t \in [0,1]} \, \Mbf(U_i(t)) &\leq C \Mbf(T_i).
\end{align*}
Here, $\dbr{F}$ is the current associated to an oriented $k$-face $F \in \Fcal_k(\eps)$ of one of the cubes $\eps z + (0,\eps)^d$ with $z \in \Z^d$, and $p^{(i)}_F \in \Z$. The constant $C > 0$ depends only on the dimensions.

Next, take any $Q \in \Irm_{k+1}(\cl{\Omega'})$ with
\[
  \partial Q = P_1 - P_0 \qquad\text{and}\qquad
  \Mbf(Q) \leq \Fbb_{\cl{\Omega'}}(P_1 - P_0) + \eps.
\]
If no such $Q$ exists, the result holds trivially since in this case $\dist_{\Lip,\cl{\Omega}}(T_0,T_1) = \Fbb_{\cl{\Omega}}(T_1 - T_0) = \infty$. So, in the following we assume the existence of at least one such $Q$.

We now apply the approximation result of Proposition~\ref{prop:approx} (which is from~\cite{ColomboDeRosaMarcheseStuvard17,ChambolleFerrariMerlet21}) to $Q$. According to this result, there is a polyhedral chain $R \in \mathrm{IP}_{k+1}(\cl{\Omega''}) \subset \Irm_{k+1}(\cl{\Omega''})$, where $\Omega'' := \Omega' + B(0,\eps) = \Omega + B(0,(\sqrt{d}+2)\eps)$, of the form
\begin{equation} \label{eq:R}
  R = \sum_\ell p_\ell \, \dbr{\sigma_\ell},
\end{equation}
where the $\sigma_\ell$ are convex $(k+1)$-polytopes (not necessarily $(k+1)$-faces of cubes as for $P_1,P_2$) and $p_\ell \in \N$, such that
\[
  \partial R = \partial Q = P_1 - P_0,  \qquad
  \Fbb_{\cl{\Omega''}}(Q - R) < \eps,  \qquad
  \Mbf(R) < \Mbf(Q) + \eps.
\]
Potentially chopping every $\sigma_\ell$ into several sub-polytopes, we may additionally assume that the $\sigma_\ell$ are disjoint up to an $\Hcal^{1+k}$-negligible set. Note that our need for the mass bound $\Mbf(R) < \Mbf(Q) + \eps$ requires the use of an approximation theorem beyond the standard deformation theorem.

We claim that there exists $V \in \Irm^\Lip_{1+k}([0,1] \times \cl{\Omega''})$ with
\[
  \partial V = - \delta_0 \times \partial R
\]
and
\[
  \Var(V) = \Mbf(R),  \qquad
  \esssup_{t \in [0,1]} \, \Mbf(V(t)) < \infty.
\]
Indeed, for every oriented $k$-polytope $\sigma_\ell$ we denote the center of $\sigma_\ell$ by $z_\ell$ and consider the \enquote{reverse cone}
\[
  \dbr{\partial \sigma_\ell} \revcone z_\ell  := \bar{H}_* (\dbr{(0,1)} \times \dbr{\partial \sigma_\ell}) \in \Irm^\Lip_{1+k}([0,1] \times \cl{\Omega''}),
\]
where $H(t,x) := (1-t)x + tz_\ell$ and $\bar{H}(t,x) := (t,H(t,x))$, similarly to Example~\ref{ex:defo_cone} (in fact, the reverse cone is indeed the reversal of the cone $z_\ell \cone \dbr{\partial \sigma_\ell}$ in the sense of Lemma~\ref{lem:reversal}). It follows that
\[
  \partial \bigl[ \dbr{\partial \sigma_\ell} \revcone z_\ell \bigr] = - \delta_0 \times \dbr{\partial \sigma_\ell}
\]
and, by Lemma~\ref{lem:injhom},
\[
  \Var(\dbr{\partial \sigma_\ell} \revcone z_\ell)
  = \Mbf \bigl( H_* (\dbr{(0,1)} \times \dbr{\partial \sigma_\ell}) \bigr)
  = \Mbf(\dbr{\sigma_\ell}).
\]
Moreover, as the $\ell$'th reverse cone shrinks to the point $z_\ell$,
\[
  \esssup_{t \in [0,1]} \, \Mbf \bigl(\bigl[ \dbr{\partial \sigma_\ell} \revcone z_\ell \bigr](t) \bigr)
  \leq \Mbf(\dbr{\partial \sigma_\ell}).
\]
Hence, setting
\[
  V := \sum_\ell p_\ell \, \bigl[ \dbr{\partial \sigma_\ell} \revcone z_\ell \bigr],
\]
we obtain 
\[
  \partial V = - \delta_0 \times \biggl( \sum_\ell p_\ell \, \dbr{\partial \sigma_\ell} \biggr)
  = - \delta_0 \times \biggl( \partial \sum_\ell p_\ell \, \dbr{\sigma_\ell} \biggr)
  = - \delta_0 \times \partial R
\]
and
\[
  \Var(V) = \sum_\ell p_\ell \, \Var( \dbr{\partial \sigma_\ell} \revcone z_\ell )
   = \sum_\ell p_\ell \, \Mbf(\dbr{\sigma_\ell})
   = \Mbf(R).
\]
We have thus constructed $V$ as required.

For $W := -V + \dbr{(0,1)} \times P_1 \in \Irm^\Lip_{1+k}([0,1] \times \cl{\Omega''})$ we compute
\[
  \partial W = \delta_0 \times \partial R + \delta_1 \times P_1 - \delta_0 \times P_1 = \delta_1 \times P_1 - \delta_0 \times P_0
\]
and
\[
  \Var(W) = \Var(V) = \Mbf(R).
\]
We now concatenate $U_0, W$ and the reversal of $U_1$ via Lemmas~\ref{lem:concat},~\ref{lem:reversal} to obtain a Lip-integral current $\tilde{S} \in \Irm^\Lip_{1+k}([0,1] \times \cl{\Omega''})$, for which it holds that $\partial \tilde{S} = \delta_1 \times T_1 - \delta_0 \times T_0$ and
\[
  \absb{\Var(\tilde{S}) - \Mbf(R)} = \Var(U_0) + \Var(U_1) \leq C \eps (\Mbf(T_0) + \Mbf(T_1)).
\]

For $\eps > 0$ suitably small there is a Lipschitz retraction $r \colon \cl{\Omega''} \to \cl{\Omega}$ with $\abs{Dr} = \BigO(\eps)$, see Remark~\ref{rem:retract}. For $S := r_* \tilde{S} \in \Irm^\Lip_{1+k}([0,1] \times \cl{\Omega})$ we then have
\[
  \partial S = \delta_1 \times T_1 - \delta_0 \times T_0
\]
and
\[
  \absb{\Var(S) - \Mbf(R)} \leq \BigO(\eps).
\]
Moreover, using the Lipschitz retraction once more,
\[
  \Fbb_{\cl{\Omega'}}(T_1 - T_0) \leq \Fbb_{\cl{\Omega}}(T_1 - T_0) + \BigO(\eps).
\]
Combining all the above estimates, we get
\begin{align*}
  \dist_{\Lip,\cl{\Omega}}(T_0,T_1) &\leq \Var(S) \\
  &\leq \Mbf(R) + \BigO(\eps) \\
  &\leq \Mbf(Q) + \BigO(\eps) \\
  &\leq \Fbb_{\cl{\Omega'}}(P_1 - P_0) + \BigO(\eps) \\
  &\leq \Fbb_{\cl{\Omega'}}(T_1 - T_0) + \BigO(\eps), \\
  &\leq \Fbb_{\cl{\Omega}}(T_1 - T_0) + \BigO(\eps),
\end{align*}
where in every line we combine the error terms by changing the expression for $\BigO(\eps)$. Letting $\eps \to 0$, we arrive at
\[
  \dist_{\Lip,\cl{\Omega}}(T_0,T_1) \leq \Fbb_{\cl{\Omega}}(T_1 - T_0).
\]
This finishes the proof.
\end{proof}

\begin{remark}
Note that in codimension $1$, i.e., $k+1 = d$, there is only one candidate surface $Q$ with $\partial Q = T_1 - T_0$ in $\Fbb_{\cl{\Omega}}(T_1 - T_0)$, up to a fixed multiple of Lebesgue measure. Indeed, if $Q_i \in \Irm_d(\cl{\Omega})$ with $\partial Q_i = T_1 - T_0$ for $i = 1,2$, then $R := Q_2 - Q_1$ is a boundaryless integral $d$-current in $\R^d$ (recall that we always use the \emph{global} boundary operator). Hence, by the constancy theorem (see, e.g.,~\cite[Theorem~7.3.1]{KrantzParks08book} or~\cite[4.1.4]{Federer69book}), $R$ is a fixed multiple of $\Lcal^d$. Thus, there is only one such surface $Q$ with globally finite mass, which we denote as $\bar{Q}$. Then, $\pbf_* S = \bar{Q}$ for any $S \in \Irm^\Lip_{1+k}([0,1] \times \cl{\Omega})$ with $\partial S = \delta_1 \times T_1 - \delta_0 \times T_0$ since any such $\pbf_* S$ has finite mass in $\R^d$ by Lemma~\ref{lem:mass}. This immediately yields the claim of the preceding theorem in this case.
\end{remark}


\begin{thebibliography}{10}

\bibitem{AbbaschianReedHill09}
R.~Abbaschian, L.~Abbaschian, and R.~E. Reed-Hill, \emph{{Physical Metallurgy
  Principles - SI Edition}}, Cengage Learning, 2009.

\bibitem{AmbrosioFuscoPallara00book}
L.~Ambrosio, N.~Fusco, and D.~Pallara, \emph{{Functions of Bounded Variation
  and Free-Discontinuity Problems}}, Oxford Mathematical Monographs, Oxford
  University Press, 2000.

\bibitem{AndersonHirthLothe17book}
P.~M. Anderson, J.~P. Hirth, and J.~Lothe, \emph{Theory of dislocations},
  Cambridge University Press, 2017.

\bibitem{BrezisMironescu19}
H.~Brezis and P.~Mironescu, \emph{The {P}lateau problem from the perspective of
  optimal transport}, C. R. Math. Acad. Sci. Paris \textbf{357} (2019),
  597--612.

\bibitem{ChambolleFerrariMerlet21}
A.~Chambolle, L.~A.~D. Ferrari, and B.~Merlet, \emph{Strong approximation in
  {$h$}-mass of rectifiable currents under homological constraint}, Adv. Calc.
  Var. \textbf{14} (2021), 343--363.

\bibitem{ColomboDeRosaMarcheseStuvard17}
M.~Colombo, A.~De~Rosa, A.~Marchese, and S.~Stuvard, \emph{On the lower
  semicontinuous envelope of functionals defined on polyhedral chains},
  Nonlinear Anal. \textbf{163} (2017), 201--215.

\bibitem{ContiGarroniMassaccesi15}
S.~Conti, A.~Garroni, and A.~Massaccesi, \emph{Modeling of dislocations and
  relaxation of functionals on 1-currents with discrete multiplicity}, Calc.
  Var. Partial Differential Equations \textbf{54} (2015), 1847--1874.

\bibitem{ContiGarroniOrtiz15}
S.~Conti, A.~Garroni, and M.~Ortiz, \emph{The line-tension approximation as the
  dilute limit of linear-elastic dislocations}, Arch. Ration. Mech. Anal.
  \textbf{218} (2015), 699--755.

\bibitem{Federer69book}
H.~Federer, \emph{{Geometric Measure Theory}}, Grundlehren der mathematischen
  Wissenschaften, vol. 153, Springer, 1969.

\bibitem{FedererFleming60}
H.~Federer and W.~H. Fleming, \emph{Normal and integral currents}, Ann. of
  Math. \textbf{72} (1960), 458--520.

\bibitem{GiaquintaModicaSoucek98book1}
M.~Giaquinta, G.~Modica, and J.~Sou{\v{c}}ek, \emph{{Cartesian Currents in the
  Calculus of Variations. {I}}}, Ergebnisse der Mathematik und ihrer
  Grenzgebiete, vol.~37, Springer, 1998.

\bibitem{GiaquintaModicaSoucek98book2}
\bysame, \emph{{Cartesian Currents in the Calculus of Variations. {II}}},
  Ergebnisse der Mathematik und ihrer Grenzgebiete, vol.~38, Springer, 1998.

\bibitem{HardtPitts86}
R.~M. Hardt and J.~T. Pitts, \emph{Solving {P}lateau's problem for
  hypersurfaces without the compactness theorem for integral currents},
  Geometric measure theory and the calculus of variations ({A}rcata, {C}alif.,
  1984), Proc. Sympos. Pure Math., vol.~44, Amer. Math. Soc., Providence, RI,
  1986, pp.~255--259.

\bibitem{HudsonRindler21?}
T.~Hudson and F.~Rindler, \emph{Elasto-plastic evolution of crystal materials
  driven by dislocation flow}, Math. Models Methods Appl. Sci. (M3AS) (2022),
  to appear, arXiv:2109.08749.

\bibitem{HullBacon11book}
D.~Hull and D.~J. Bacon, \emph{{Introduction to Dislocations}}, 5 ed.,
  Elsevier, 2011.

\bibitem{JungJerrard04}
R.~L. Jerrard and N.~Jung, \emph{Strict convergence and minimal liftings in
  {$BV$}}, Proc. Roy. Soc. Edinburgh Sect. A \textbf{134} (2004), 1163--1176.

\bibitem{Kampschulte17PhD}
M.~Kampschulte, \emph{{Gradient flows and a generalized Wasserstein distance in
  the space of Cartesian currents}}, Ph.D. thesis, RWTH Aachen University,
  2017.

\bibitem{KrantzParks08book}
S.~G. Krantz and H.~R. Parks, \emph{Geometric integration theory},
  Birkh\"auser, 2008.

\bibitem{LuukkainenVaisala77}
J.~Luukkainen and J.~V\"{a}is\"{a}l\"{a}, \emph{Elements of {L}ipschitz
  topology}, Ann. Acad. Sci. Fenn. Ser. A I Math. \textbf{3} (1977), 85--122.

\bibitem{Rindler21b?}
F.~Rindler, \emph{Energetic solutions to rate-independent large-strain
  elasto-plastic evolutions driven by discrete dislocation flow}, Preprint,
  arXiv:2109.14416.

\bibitem{RindlerShaw19}
F.~Rindler and G.~Shaw, \emph{Liftings, {Y}oung measures, and lower
  semicontinuity}, Arch. Ration. Mech. Anal. \textbf{232} (2019), 1227--1328.

\bibitem{ScalaVanGoethem19}
R.~Scala and N.~Van~Goethem, \emph{Variational evolution of dislocations in
  single crystals}, J. Nonlinear Sci. \textbf{29} (2019), 319--344.

\end{thebibliography}

\providecommand{\bysame}{\leavevmode\hbox to3em{\hrulefill}\thinspace}
\providecommand{\MR}{\relax\ifhmode\unskip\space\fi MR }
\providecommand{\MRhref}[2]{%
  \href{http://www.ams.org/mathscinet-getitem?mr=#1}{#2}
}
\providecommand{\href}[2]{#2}

\end{document}